\documentclass[a4paper,10pt]{scrartcl}

\usepackage{amsmath,amssymb,amsfonts}
\usepackage{ifthen}
\usepackage{hyperref}
\usepackage[amsmath,thmmarks,hyperref]{ntheorem}
\usepackage[boxed]{algorithm2e}
\usepackage[english]{babel}
\usepackage{todonotes}
\usepackage[affil-it]{authblk}

\theoremstyle{plain}
\theoremseparator{.}
\newtheorem{theorem}{Theorem}

\newtheorem{lemma}[theorem]{Lemma}

\theorembodyfont{}
\theoremsymbol{\ensuremath{{}_\blacksquare}}

\theoremheaderfont{\itshape}
\newtheorem{remark}[theorem]{Remark}

\theoremsymbol{\ensuremath{\square}}
\newtheorem*{proof}{Proof}

\DeclareMathOperator{\proj}{proj}

\newcommand{\field}[1]{\mathbb{#1}}
\newcommand{\R}{\field{R}}

\DeclareMathOperator{\dom}{dom}

\newcommand{\abs}[2][n]{\SwitchBracketsizeLeft{#1}\LeftBracketSize\lvert#2\SwitchBracketsizeRight{#1}\RightBracketSize\rvert}
\newcommand{\norm}[2][n]{\SwitchBracketsizeLeft{#1}\LeftBracketSize\lVert#2\SwitchBracketsizeRight{#1}\RightBracketSize\rVert}

\newcommand{\NextScriptStyle}[1]{{\scriptstyle{#1}}}
\newcommand{\NextScriptScriptStyle}[1]{{\scriptscriptstyle{#1}}}
\newcommand{\NextTextStyle}[1]{{\textstyle{#1}}}
\newcommand{\NextDisplayStyle}[1]{{\displaystyle{#1}}}
\newcommand{\SwitchBracketsizeLeft}[1]{
  \ifthenelse{\equal{#1}{b}\OR\equal{#1}{big}}{\let\LeftBracketSize=\bigl}{
    \ifthenelse{\equal{#1}{B}\OR\equal{#1}{Big}}{\let\LeftBracketSize=\Bigl}{
      \ifthenelse{\equal{#1}{g}\OR\equal{#1}{bigg}}{\let\LeftBracketSize=\biggl}{
    \ifthenelse{\equal{#1}{G}\OR\equal{#1}{Bigg}}{\let\LeftBracketSize=\Biggl}{
      \ifthenelse{\equal{#1}{s}\OR\equal{#1}{small}}{\let\LeftBracketSize=\NextScriptStyle}{
        \ifthenelse{\equal{#1}{ss}}{\let\LeftBracketSize=\NextScriptScriptStyle}{
          \ifthenelse{\equal{#1}{t}\OR\equal{#1}{text}}{\let\LeftBracketSize=\NextTextStyle}{
        \ifthenelse{\equal{#1}{d}\OR\equal{#1}{display}}{\let\LeftBracketSize=\NextDisplayStyle}{
          \ifthenelse{\equal{#1}{a}\OR\equal{#1}{auto}}{\let\LeftBracketSize=\left}{
            \let\LeftBracketSize=\relax}}}}}}}}}}
\newcommand{\SwitchBracketsizeRight}[1]{
  \ifthenelse{\equal{#1}{b}\OR\equal{#1}{big}}{\let\RightBracketSize=\bigr}{
    \ifthenelse{\equal{#1}{B}\OR\equal{#1}{Big}}{\let\RightBracketSize=\Bigr}{
      \ifthenelse{\equal{#1}{g}\OR\equal{#1}{bigg}}{\let\RightBracketSize=\biggr}{
    \ifthenelse{\equal{#1}{G}\OR\equal{#1}{Bigg}}{\let\RightBracketSize=\Biggr}{
      \ifthenelse{\equal{#1}{s}\OR\equal{#1}{small}}{\let\RightBracketSize=\NextScriptStyle}{
        \ifthenelse{\equal{#1}{ss}}{\let\RightBracketSize=\NextScriptScriptStyle}{
          \ifthenelse{\equal{#1}{t}\OR\equal{#1}{text}}{\let\RightBracketSize=\NextTextStyle}{
        \ifthenelse{\equal{#1}{d}\OR\equal{#1}{display}}{\let\RightBracketSize=\NextDisplayStyle}{
          \ifthenelse{\equal{#1}{a}\OR\equal{#1}{auto}}{\let\RightBracketSize=\right}{
            \let\RightBracketSize=\relax}}}}}}}}}}

\makeatletter
\newcommand{\logmessage}[1]{\@latex@warning{#1}}
\newcommand{\ignore}{\logmessage{Text ignored}\@gobble}
\makeatother

\DeclareMathOperator{\vol}{vol}
\DeclareMathOperator{\trace}{Tr}

\date{\today}

\author{Martin Bauer\thanks{Faculty of Mathematics, University of Vienna.}, Markus Grasmair\thanks{Department of Mathematical Sciences, NTNU Trondheim.}, Clemens Kirisits\thanks{Computational Science Center, University of Vienna.}}

\title{Optical Flow on Moving Manifolds}

\begin{document}

\maketitle
\begin{abstract}
{\bf Abstract.} 
 Optical flow is a powerful tool for the study and analysis of motion in a sequence of images. 
 In this article we study a Horn--Schunck type spatio-temporal regularization functional 
 for image sequences that have a non-Euclidean, time varying image domain. 
 To that end we construct a Riemannian metric that describes the deformation and structure of this evolving surface.  
 The resulting functional can be seen as natural geometric generalization of previous work by Weickert and Schn\"orr (2001) 
 and Lef\`evre and Baillet (2008) for static image domains. In this work we show the existence and wellposedness of 
 the corresponding optical flow problem and derive necessary and sufficient optimality conditions. 
 We demonstrate the functionality of our approach in a series of experiments using both synthetic and real data. 

{\bf Keywords.} Optical flow, evolving surfaces, spatio-temporal regularization, variational methods.

{\bf MSC Subject Classification.} 
58J90; 
65M30; 
68U10; 
\end{abstract}

\section{Introduction}
\paragraph{Optical flow.}
Optical flow is a powerful tool for detecting and analyzing motion in a sequence of images. 
The underlying idea is to depict the displacement of patterns in the image sequence 
as a vector field --- the optical flow vector field --- 
generating the corresponding displacement function.
This framework has applications in a variety of areas connected to computer graphics and 
video analysis, e.g. in video compressing, video surveillance or vision-based robot navigation.  

\paragraph{Variational methods.}
In their seminal article \cite{HorSchu81}, Horn and Schunck proposed 
a \emph{variational ansatz} for the computation of the optical flow vector field.
In this approach one minimizes an 
\emph{energy functional} consisting of a \emph{similarity (data) term} 
and a \emph{regularity term}:
$$\tilde u=\underset{u\in \mathcal H}{\operatorname{argmin}} \left. \mathcal E(u) \right. = \underset{u\in \mathcal H}{\operatorname{argmin}}  \left(\mathcal S(u)+\mathcal R(u)\right)\,. $$
Here $\mathcal H$ denotes an admissible space of vector fields, $\mathcal R$ 
denotes the regularity term for the vector field $u$, and $\mathcal S$ 
denotes the similarity term, which depends on the image sequence $\mathcal I$ under consideration. 
This method turned out to be particularly successful, as the resulting optical flow fields satisfy certain 
desirable properties governed by the choice of the regularization term $\mathcal{R}$.

In their article, Horn and Schunck considered the optical flow problem for 
a sequence of images defined on some domain in $\R^2$.
They proposed  to use the $L^2$-norm of the first derivative of the vector field $u$ 
as a regularization term.  The \emph{wellposedness} of this ansatz has been shown first
by  Schn\"orr in \cite{Schn91a}. There they had to impose an additional assumption on 
the image sequence in order to ensure the coercivity of the functional $\mathcal E$;
this is mainly caused by the so called \emph{aperture problem}, which results from
the impossibility of detecting or discriminating certain types of motion 
in a very regular image.

After the development of the Horn--Schunck functional, several extensions 
and improvements of the regularization term have been developed, see e.g.
\cite{AubDerKor99,BruWeiSchn05,Coh93,Nagel1987,Nagel1988,NagEnk86}. A survey on variational techniques for optical flow can be found in \cite{WeiBruBroPap06}. In the article \cite{Nagel1990}, Nagel proposed 
to add \emph{regularization in time} via smoothing across image discontinuities.
\emph{Time smoothing of the flow field}, on the other hand,
has been introduced by Weickert and Schn\"orr in \cite{WeiSchn01}, where the authors considered an 
additional term containing the time derivative of the vector 
field $u$ in the definition of the regularization functional.  
This alteration still yields a convex  energy functional and 
thus the wellposedness of the optical flow problem can be proven 
employing similar methods as for the original Horn--Schunck functional. 
While these results have been derived for domains in $\R^2$, the
situation of more general --- possibly curved --- image domains has not been considered there.
A first attempt in this direction can be found in \cite{ImiSugTorMoc05,TorImiSugMoc05},
where the authors introduced the optical flow functional for images  on the round sphere. Finally, the case of an arbitrary compact
two-dimensional manifold as image domain has been studied in \cite{LefBai08}. There
the authors discuss the usage of the Horn--Schunck functional on a manifold and
prove a similar wellposedness result as for the plane.

\paragraph{Time varying image domains.}
Recently, Kirisits, Lang and Scherzer have studied the optical flow problem
on a time varying image domain~\cite{KirLanSch13,KirLanSch14a}. The motivation for that
was an application in volumetric microscopy, where one studies
the early development of a zebra-fish embryo. In this setting, almost all
movement between consecutive images takes place on the surface
of the embryo's yolk cell, which, however, is time-dependent as well.
In theory, it would be possible to use the complete volumetric data
in order to compute a three-dimensional optical flow field. In practice,
however, this is not viable because of the huge amount of data involved.
Instead, it makes sense to extract the moving surface in a first step
and then to compute the flow field on this surface in a second, separate step.

The main mathematical challenge at this point is the correct treatment
of a vector field on a moving manifold $\mathcal{M}_t \subset \R^3$, $t \in [0,T]$.
We assume in this paper that this manifold is given by a family of parametrizations
$f(t,\cdot) \colon M\to \R^3$, where the configuration space $M$ is a fixed compact
two-dimensional manifold (possibily with boundary).
The image sequence is defined on this moving manifold, and it is assumed that the structure
of the manifold has an influence on the deformation of the image sequence. The difficulty is  
\emph{to capture the structure of the moving manifold in the optical flow field.} 
Therefore, one has to develop a regularization term that
depends on \emph{the induced, changing Riemannian metric.}

At this point, we want to note that it would, in principle, be possible
to use some fixed Riemannian metric on $M$ in order to obtain a regularization term
like in~\cite{LefBai08}. Then one would lose, however, all the information
about the correct manifold $\mathcal{M}_t$ as well as its movement in space.
In the experimental section we will compare this naive approach to our geometrical 
method and we will see that there is a significant 
difference in the resulting optical flow fields.

\paragraph{Contributions of the article.}
One possibility of a regularization term capturing the structure
of a moving manifold has already been given
in~\cite{KirLanSch13, KirLanSch14a}. In this article we propose a different one
that is induced by a metric $\bar g$ on the product manifold $\bar M=[0,T]\times M$. This metric 
$\bar g$ is constructed in such a way that it incorporates all available information 
on the moving image domain:
\[
\bar g(\cdot,\cdot)=
\begin{pmatrix}
\alpha^2&0\\
0&f^*\langle\cdot,\cdot\rangle_{\R^3}
\end{pmatrix}.
\]
The constant $\alpha>0$ is a weighting parameter
and $f^*\langle\cdot,\cdot\rangle_{\R^3}$ denotes the induced 
surface metric of the parametrization $f$ at time $t$.
Given such a metric, we can use a weighted $H^1$-norm as regularization term:
$$\mathcal R(\bar u)=\int_{\bar M}\beta \bar g(\bar u,\bar u)+\gamma \bar g^1_1(\bar \nabla \bar u,\bar\nabla \bar u) \operatorname{vol}(\bar g) \,.$$
This regularization term is defined for vector fields $\bar{u}$ on the product manifold $\bar{M}$. However, since we do not want to change the time parametrization, we will only consider 
vector fields with vanishing time component, cf.\ Remark~\ref{rem:zeropartialt} for a more detailed explanation of this choice.
Moreover, $\bar g^1_1$ denotes the extension of the metric to $1$-$1$ tensor fields, $\bar \nabla$ denotes the covariant 
derivative of $\bar g$ and $\operatorname{vol}(\bar g)$ is the corresponding volume form. 
Note that this term enforces spatio-temporal regularity, as it contains derivatives in both time and space. 
The parameter $\alpha$ that is included in the definition of the metric allows to 
penalize regularity in time and space separately.
This choice for the regularization term is a natural geometric generalization of the
regularization term on the static manifold $[0,T]\times \R^2$ from \cite{WeiSchn01}.

If we decide to enforce no regularity in time, then the optical flow problem reduces for each time point $t_i$
to the optical flow problem on the static manifold $\mathcal M_{t_i}$. In this case,
our regularization term equals the regularization term used in \cite{LefBai08}.

The similarity term we use in this paper is simply the
squared $L^2$-norm of the defect of the optical flow equation, that is,
\[
\mathcal{S}(\bar u) = \int_{\bar{M}} (\partial_t I + g(\nabla^g I,u))^2\operatorname{vol}(\bar g)\,.
\]
Regarding the wellposedness of this optical flow problem we obtain the following result:

\begin{theorem}[Wellposedness of the optical flow problem]
  Let $\mathcal{M}_t \subset \R^3$ be a moving two-dimensional compact surface, 
  the movement of which is described by a family of parametrizations $f\colon M\to \mathbb R^3$.
  For all parameters $\beta, \gamma > 0$ and any image sequence $I \in
  W^{1,\infty}(\bar{M})$ the optical flow
  functional 
  \[
  \mathcal{E}(\bar{u}) = \mathcal{S}(\bar{u}) + \mathcal{R}(\bar{u})
  \]
  has a unique minimizer in $$\dom(\mathcal{E}):=\{\bar u\in H^1(\bar M,T\bar M ): 
  \bar u =(0\partial_t, u) \text{ with } u\in H^1(\bar{M},TM) \}\,.$$
\end{theorem}
A similar result is also shown under the assumption of partial Dirichlet boundary conditions.
The $L^2$-norm in the regularization term is added to enforce the coercivity of the energy functional. We also discuss under which assumptions we can 
set the parameter $\beta$ to zero and still obtain a wellposedness result for our functional. We compare our functional to the functionals introduced
by Kirisits et al.~\cite{KirLanSch13,KirLanSch14a} and discuss the wellposedness of the optical flow problem using the regularization terms that are employed there.

Finally, we demonstrate the functionality of our approach in a series of experiments using both synthetic and real data. 
In these experiments we also show the difference between our approach and the straightforward approach, that does not use the actual 
structure of the moving manifold. In both experiments one can see notably different results in regions of
the manifold where either the curvature or the deformation of the manifold is large. 

Another important topic is the strong dependence of the 
optical flow field on the parametrization of the moving manifold. In appendix~\ref{ap:param} we present 
a brief discussion on a possible approach to compute realistic parametrizations given an observed moving un--parametrized manifold. The long term goal will be the combination of segmentation and computation
of the optical flow, which we hope will lead to more reliable results.

\paragraph{Organization of the article.}
In Section~\ref{preliminaries} we recall the differential geometric and functional analytic tools that we will use throughout the article. Readers that are acquainted 
with the theory of Sobolev spaces of vector fields on Riemannian manifolds might skip this part and directly start with Section~\ref{problemformulation} which contains the rigorous mathematical 
formulation of the  optical flow problem studied in this article. In Section~\ref{regularization} we construct the regularization term that we employ in this article, prove the
wellposedness of the corresponding functional and derive the optimality conditions. Up to this point all calculations and results are presented in a coordinate independent 
manner. In order to obtain an implementable version we derive in Section~\ref{coordinates} a coordinate version of the optimality conditions. 
This involves rather technical calculations, that are partly postponed to the appendix.
In Section~\ref{experiments} we show   numerical experiments that demonstrate the functionality of the proposed 
energy functional. 
The appendix contains a discussion on how to compute the parametrization of the moving manifold
and the actual calculations of the coordinate version of the optimality conditions.

\paragraph{Acknowledgements.}
We thank Pia Aanstad from the University of Innsbruck for kindly providing the microscopy data used in this article.
We are grateful for helpful comments and discussions to Lukas Lang, Peter Michor and Otmar Scherzer.  
Martin Bauer was supported by FWF project P24625 and Clemens Kirisits was supported by the Vienna Graduate School in Computational Science (IK I059-N) and FWF project S10505-N20.
Finally we want to thank the referees for their careful proofreading and their valuable comments.

\section{Mathematical Preliminaries}\label{preliminaries}

In this section we are going to recall the differential geometric
and functional analytic tools for Sobolev spaces of vector fields
on $2$-dimensional embedded surfaces, which we will need throughout the article.
A more detailed overview on these topics can be found e.g.~in \cite[Sect.~3]{Bauer2012a}.

\paragraph{Riemannian geometry.}

We are working on $2$-dimensional surfaces that are embedded in $\R^3$ and
parametrized by a mapping
$$f: M\to \R^3$$
from some \emph{configuration space} $M$ into $\R^3$. 
We will always assume that $M$ is a compact $2$-dimensional manifold,
possibly with boundary; typical examples
are the $2$-dimensional sphere $S^2$ or the torus $S^1 \times S^1$.
The mapping $f$ is assumed to be smooth (that is, at least $C^2$) and injective
with injective tangential mapping $Tf\colon TM \to \R^3$
(in other words, $f$ is a smooth embedding).

The embedding $f$ induces in a natural way via pullback a Riemannian metric $g$
on the configuration space $M$. For tangent vectors $X$, $Y\in T_xM$, $x \in M$,
it is given by
$$g(X,Y):=(f^*\langle \cdot,\cdot\rangle_{\R^3})(X,Y) := \langle
T_xf.X,T_xf.Y\rangle_{\R^3}\,.$$
Here $\langle \cdot,\cdot\rangle_{\R^3}$ denotes the standard scalar product on $\R^3$
and $T_xf.X$ denotes the application of the differential $T_xf: T_xM \rightarrow \mathbb R^3$ of the embedding $f$ to 
the tangent vector $X\in T_xM$.

In a chart $(V,v)$ on $M$ the expression of the metric reads as
$$g|_V=\sum_{i,j}g_{ij}dv^i \otimes
dv^j=\sum_{i,j}\langle \partial_if,\partial_j f \rangle_{\R^3} dv^i \otimes dv^j$$
with $\partial_i=\frac{\partial}{\partial v^i}$. 

Next we note that the metric induces an isomorphism between the tangent bundle 
and the cotangent bundle
$$\check g\colon TM \to T^*M,\qquad X \mapsto g(X,\cdot):= X^\flat\;,$$
with inverse $\check g^{-1}$.
Therefore $g$ defines a metric on the cotangent bundle $T^*M$ via
\[
g^{-1}(\alpha,\beta) = \alpha(\check g^{-1}(\beta))\,.
\]
In this article we will need the extension of the metric to $1$-$1$ tensor fields. The reason for this is, that 
this type of tensor field occurs as derivative of a vector field on $M$, which will be a part of our regularization term.
On these tensor fields the metric is given by $$g_1^1:=g\otimes g^{-1}.$$ 

Applied to a $1$-$1$ tensor field $A$ 
this equals the squared Hilbert Schmidt norm of $A$:
\[
 g^1_1( A, A) = \trace(A^* A),
\]
where the adjoint $A^*$ is computed with respect to the
Riemannian metric $g$. 
Here we have interpreted $A$ as a linear
mapping from $T_{x} M$ to $T_{x} M$.

%
%


\paragraph{Sobolev spaces of vector fields.}
The Riemannian metric $g$ on $M$ induces a unique volume density, which we will 
denote by $\vol(g)$. In the chart 
$(V,v)$ its formula reads as
$$\operatorname{vol}(g)|_{V}=\sqrt{\operatorname{det}(\langle \partial_if,\partial_j f
\rangle_{\R^3})}\ |dv^1\wedge dv^2|\,.$$

The Levi--Civita covariant derivative of the metric $g$, which is the unique torsion-free 
connection preserving the metric $g$, will be denoted by $\nabla^g$. When it is clear from the context, we omit the $g$ and simply write $\nabla$ instead of $\nabla^g$.
Note that $\nabla^g$ is just the tangential component of the usual derivative
in the ambient space $\mathbb R^3$, more precisely, for a vector field $u\in C^{\infty}(M,TM)$
$$Tf.\nabla^g u = \proj_{T^{1,1}M} \nabla^{\R^3}(Tf.u)\,.$$

We define 
the Sobolev norms of orders zero and one by
\begin{align*}
	\|u\|^2_{0,g}	&= \int_M g(u,u) \operatorname{vol}(g)\,, \\
	\|u\|^2_{1,g}	&= \int_M g(u,u)+g^1_1(\nabla^g u,\nabla^g u)\operatorname{vol}(g)\,.
\end{align*}

The Sobolev space $H^1(M,TM)$ is then defined as the completion of the space of all 
vector fields $u\in C^\infty(M,TM)$ with respect to the norm $\|\cdot\|_{1,g}$. On a compact manifold different metrics yield equivalent norms and thus lead to the same Sobolev spaces.
We note that, as for Sobolev spaces in $\R^n$, there is an alternative, equivalent
definition of $H^1(M,TM)$ as the space of all square integrable vector fields with
square integrable weak derivatives.

For the definition of more general Sobolev space on manifolds we refer 
to \cite{Triebel1992}; see also \cite{Bauer2013b} for an exposition in a similar notation 
as it is used in this article.
An extension of this theory  to non-compact manifolds can be found in the book \cite{Eichhorn2007}.

\section{Problem Formulation}\label{problemformulation}

We assume that we are given a moving two-dimensional compact surface
$\mathcal{M}_t \subset \R^3$, $t \in [0,T]$ the movement of which is
described by a family of parametrizations $f$. For the moment, we will
restrict ourselves to compact surfaces without boundary, but we will
discuss the situation of manifolds with boundary later.
More precisely, we assume that there
exists a two-dimensional compact $C^2$-manifold $M$ and a $C^2$-mapping
\[
f\colon [0,T] \times M \to \R^3
\]
such that for every fixed time $t \in [0,T]$ the mapping $f(t,\cdot)$
is an embedding and its image equals $\mathcal{M}_t$.
The mapping $f$ defines the movement
of the manifold in the sense that the path of a point $y = f(0,x) \in
\mathcal{M}_0$ is precisely the curve $t \mapsto f(t,x)$. Or, a
point $y_1 \in \mathcal{M}_{t_1}$ corresponds to a point $y_2 \in
\mathcal{M}_{t_2}$ if and only if there exists $x \in M$ with $y_1 =
f(t_1,x)$ and $y_2 = f(t_2,x)$.

Next we model the movement of an image on this moving surface. 
For simplicity we will only consider grey-scale images, although the
model does not change significantly if we also allow color, that is,
vector valued, images. We stress that in our model the movement of
the image is not solely driven by the movement of the surface, but
that there is also an additional movement on the surface, the
reconstruction of which is precisely what we are aiming for.

The image sequence we are considering is given by a real valued function
$\mathcal{I}$ on
\[
\mathcal{M}:=\bigcup_{0\le t \le T} \{t\}\times \mathcal{M}_t \subset
[0,T] \times \R^3\,;
\]
for each $t\in[0,T]$, the function
$\mathcal{I}(t,\cdot)\colon\mathcal{M}_t \to \R$ is the image at the
time $t$. Moreover, there exists a family of diffeomorphisms
$\psi(t,\cdot) \colon \mathcal{M}_0\to\mathcal{M}_t$ such that
\[
\mathcal{I}(0,x) = \mathcal{I}(t,\psi(t,x)).
\]
That is, the diffeomorphisms $\psi(t,\cdot)$ generate the movement of
the image on the evolving surface.

Next, it is possible to pull back the image and the driving family of
diffeomorphisms to the configuration space $M$. Doing so, we obtain a
time dependent function $I\colon[0,T]\times M \to \R$ defined by
\[
I(t,x) = \mathcal{I}(t,f(t,x))
\]
and a family of diffeomorphisms $\varphi(t,\cdot)$ of $M$ defined by
\[
f(t,\varphi(t,x)) = \psi(t,f(0,x))
\]
such that
\begin{equation}\label{eq:Iphi}
I(0,x) = I(t,\varphi(t,x))
\end{equation}
for all $t\in[0,T]$ and $x \in M$.

Furthermore we assume that the diffeomorphisms $\varphi(t,\cdot)$ are
generated by a time dependent vector field $u$ on $M$. Then the curves
$t \mapsto \varphi(t,x)$ are precisely the integral curves of $u$,
that is, 
\begin{equation}\label{eq:dtphi}
\partial_t\varphi(t,x) = u(t,\varphi(t,x))\,.
\end{equation}
If the image $I$ is sufficiently smooth, it is possible to compute the
time derivative of equation~\eqref{eq:Iphi}. Using
the relation~\eqref{eq:dtphi} and the fact that $\varphi(t,\cdot)$ is
surjective, we then see that the image $I$ and the
vector field $u$ satisfy the \emph{optical flow equation}
\begin{equation}\label{eq:of}
0 = \partial_t I(t,x) + D_x I(t,x) u(t,x)
\end{equation}
on $[0,T] \times M$.
We do note that in the equation~\eqref{eq:of} all information about
the movement of the manifold is suppressed, as all the functions have
been pulled back to $M$. It is, however, possible to re-introduce some
knowledge of $\mathcal{M}$ by formulating the optical flow equation not
in terms of differentials but rather in terms of gradients.
To that end we denote by $g$ the time dependent Riemannian metric on $M$ that is
induced by the family of embeddings $f(t,\cdot)$.
Since by definition $\nabla^g I(t,\cdot)^\flat = D_xI(t,\cdot)$,
we can rewrite the optical flow equation as
\begin{equation}\label{eq:ofm}
  0 = \partial_t I(t,x) + g(\nabla^g I(t,x),u(t,x))
\end{equation}
for all $(t,x) \in [0,T] \times M$.

Now assume that the model manifold $M$ is a compact manifold with boundary.
Then the same model of a moving image on the embedded manifolds $\mathcal{M}_t$
is possible, as long as it is guaranteed that the boundary of the manifold
acts as a barrier for the movement of $\mathcal{I}$. That is, the diffeomorphisms
$\varphi(t,\cdot)$ satisfy the additional boundary condition
$\varphi(t,x) = x$ for $x \in \partial M$. In this case, one arrives
at the same optical flow equation~\eqref{eq:ofm}, but, additionally,
one obtains (partial) Dirichlet boundary conditions of the form
$u(t,x) = 0$ for all $(t,x) \in [0,T] \times \partial M$.

The situation is different, when the image $\mathcal{I}$ actually
moves across the boundary of $\mathcal{M}_t$, which can occur if the manifold with
boundary $\mathcal{M}_t$ represents the limited field of view on a larger manifold
that contains the moving image. Then it is not reasonable to model the
movement of the image by a family of global diffeomorphism $\psi(t,\cdot)$.
However, locally it can still be modeled as being generated by a family
of local diffeomorphisms, which in turn can be assumed to be generated
by a time dependent vector field on $M$. With this approach, one arrives,
again, at the same optical flow equation~\eqref{eq:ofm}. 
The difference to the situations discussed above is that the integral curves 
of $u$ may be defined only on bounded intervals.

\paragraph{The Inverse Problem.}
Now we consider the inverse problem of reconstructing the
\emph{movement} of a family of images from the image sequence. We
assume that we are given the family of manifolds $\mathcal{M}_t$
together with the parametrizations $f(t,\cdot)$ and the family of
images $I(t,\cdot)$ (already pulled back to $M$). Our task is to find
a time dependent vector field $u$ on $M$ that generates the movement
of $I$; in other words, a vector field $u$ that satisfies the optical
flow equation~\eqref{eq:ofm}.

Solving this equation directly is not sensible, as, in general, the
solution, if it exists, will not be unique: The optical flow equation
does not ``see'' a flow that is tangential to the level lines of the
image $I$. Thus, if $u$ is any solution of~\eqref{eq:ofm} and the
vector field $w$ satisfies
\[
g(\nabla^g I(t,x),w(t,x)) = 0,
\]
then also $u+w$ is a solution;
this is called the \emph{aperture problem} (see~\cite{FleWei06}).
In addition, the whole model fails in the case of noise leading to
non-differentiable data $I$. In order to be still able to formulate
the optical flow equation, it is possible to pre-smooth the image $I$,
but this will invariably lead to errors in the model and thus the
optical flow equation will only be satisfied approximately by the
generating vector field $u$.
For these reasons, it is necessary to introduce some kind of
regularization. Note that the main focus lies here in the problem of
solution selection.

\section{Classical Variational Regularization}\label{regularization}

One of the most straightforward regularization methods is the
application of Tikhonov regularization, where we try to minimize a
functional composed of two terms, a similarity term, which ensures
that the equation is almost satisfied, and a regularity term, which
ensures the existence of a regularized solution and is responsible for
the solution selection.

\subsection{Spatial Regularization}
If we  consider only spatial regularity,
the definition of the regularity term is straightforward, using
for each time point $t$ the pullback metric $$g(t)(\cdot,\cdot)=f(t,\cdot)^*\langle\cdot,\cdot\rangle_{\R^3}\,.$$
This leads to the energy functional
\begin{equation*}\label{eq:E_notime_def}
\boxed{
\begin{aligned}
\mathcal{E}(u) &:= \mathcal{S}(u) + 
\mathcal{R}(u) \\
&=\int_0^T \int_M
\left( \partial_t I+g(\nabla^g I,u)\right)^2
+ \beta g(u,u)
 + \gamma g^1_1( \nabla^g u, \nabla^g u) \operatorname{vol}(g)\,dt\,,
\end{aligned}
}
\end{equation*}
here $\beta$ and $\gamma$ are weighting parameters.
In this case the problem completely decouples in space and time, i.e., the optimal vector field $u$ has to be minimal
for each time point separately.  Thus the problem  reduces for each time $t$ to the
calculation of the optical flow on the (static) Riemannian manifold
$(M,g(t))$, which yields for each time point $t$ the Energy functional
\begin{equation*}\label{eq:E_notime_def2}
\boxed{
\begin{aligned}
\mathcal{E}(u(t,\cdot)) &:= \int_M
\left( \partial_t I+g(\nabla^g I,u)\right)^2
+ \beta g(u,u)
 + \gamma g^1_1( \nabla^g u, \nabla^g u) \operatorname{vol}(g)\,.
\end{aligned}
}
\end{equation*}
This functional is well
investigated. 
For  $\beta>0$ the coercivity of the energy functional
is clear and one can easily deduce the well-posedness of the optical flow problem.
In \cite{LefBai08,Schn91a} it was shown that one can  guarantee the coercivity of the energy functional 
for $\beta=0$ by requiring the image sequence to satisfy additional conditions.
The conditions in~\cite{Schn91a} for optical flow in the plane require that
the partial derivatives of the image $I$ are linearly independent
functions. This is equivalent to the requirement that no non-trivial constant vector
field $u$ satisfies the optical flow equation for the given image.
Similar requirements are commonly found for Tikhonov regularization with
derivative based regularization terms, see e.g.~\cite{AcaVog94,AubVes97,Gra11a}
and~\cite[Section~3.4]{SchGraGroHalLen09}.
In the case of a non-flat manifold $M$, the condition translates
to the non-existence of a non-trivial covariantly constant vector field
satisfying the optical flow equation.
Obviously, this condition is automatically satisfied, if the
only covariantly constant vector field is $u=0$,
and thus it may be omitted in manifold settings, see~\cite{LefBai08}.

\subsection{Regularization in time and space}
In the following we will look for solutions that additionally
satisfy a regularity constraint in time $t$. For the optical flow in
the plane $\R^2$, this method has been introduced in~\cite{WeiSchn01}.
Spatio-temporal regularization of the optical flow on moving manifolds,
has also been considered in~\cite{KirLanSch13},
but with a different regularity term than the one we will
construct in the following.

In order to construct the regularity term, we consider the product
manifold 
\[
\bar M:=[0,T]\times M
\]
and equip it with the almost product metric 
\[
\bar g(\cdot,\cdot)=
\begin{pmatrix}
\alpha^2&0\\
0&f^*\langle\cdot,\cdot\rangle_{\R^3}
\end{pmatrix}.
\]
The parameter $\alpha>0$ is a weighting parameter, which is included
in order to be able to penalize spatial regularity and regularity in
the time variable differently.
This metric is called almost product metric due to the dependence of
the metric $g(\cdot,\cdot)=f^*\langle\cdot,\cdot\rangle_{\R^3}$ on the time
$t$. In order to simplify notation we denote $\nabla^{\bar g}$ by $\bar \nabla$ from now on.

\begin{remark}
In the following, we will always indicate by a ``bar'' ($\bar{\cdot}$)
that an object is related to the product manifold $\bar{M}$.
For instance, $\bar{g}$ denotes a metric on $\bar{M}$, whereas
$g$ denotes a (time dependent) metric on $M$. Similarly, $\bar{u}$
will later denote a vector field on $\bar{M}$, whereas $u$ will
denote a time dependent vector field on $M$.
\end{remark}

\begin{remark}
We could also consider $\mathcal{M}$ as an embedded submanifold of $\R\times \R^3$:
\begin{equation*}
 \bar f\colon \left\{
\begin{array}{ccc}
[0,T]\times M&\to& \R\times \R^3 \,, \\
(t,x)&\mapsto&(t,f(t,x))\,.
\end{array} \right.
\end{equation*}
We stress here that the metric $\bar{g}$ is \emph{not} the pullback
of the (time scaled) Euclidean metric on $\mathcal{M}$ by the
parametrization $\bar f$.
Instead, it is constructed in such a way that the paths of the points on $M$
are at each time $t$ orthogonal to the manifold $\mathcal{M}_t$.
Moreover, these paths are geodesics with respect to $\bar{g}$.
These properties do, in general, not hold for the usual pullback metric.
\end{remark}

From now on we will identify a time dependent vector field $u$ on $M$
with the vector field
\[
\bar{u}(t,x) := (0\partial_t, u(t,x)) \in C^\infty(\bar{M},T\bar{M})
\]
and define both the similarity term and the regularization term in
terms of $\bar{u}$.
Taking the squared $L^2$-norm with respect
to the metric $\bar g$ of the right hand side of the optical flow
equation \eqref{eq:of} we obtain for the similarity term the
functional
\begin{align*}
\mathcal{S}(\bar{u}) 
&= \lVert \partial_t I+g(\nabla^g I,u)\rVert_{0,\bar{g}}^2\\
&=  \int_{\bar M} \left(\partial_t I(t,x) +
g(\nabla^g I(t,x),u(t,x))\right)^2 \vol(\bar g)\\& = \alpha \int_0^T \int_{M} \left(\partial_t I(t,x) +
g(\nabla^g I(t,x),u(t,x))\right)^2 \vol(g)\,dt\;.
\end{align*}
Here we used the fact that the volume form on $\bar{M}$ splits into
$\vol(\bar g)=\alpha \vol(g)\,dt$. 

For the regularization term we use a weighted $H^1$-norm of the vector field
$\bar{u}$, that is,
\begin{align}\label{Rdef}
 \mathcal{R}(\bar u)
 = \beta\lVert \bar{u}\rVert^2_{0,\bar g}
 + \gamma \|\bar\nabla\bar u\|^2_{0,\bar g} \;,
\end{align}
where $\beta$ and $\gamma$ are weighting parameters. In~\eqref{Rdef}, the term $\lVert\bar{u}\rVert^2_{0,\bar g}$
denotes the $L^2$-norm of the vector field $\bar{u}$ and $\|\bar
\nabla \bar u\|^2_{0,\bar g}$ denotes the norm of its derivative
with respect to the Riemannian metric $\bar g$ , that is, 
\begin{align*}
  \|\bar u\|^2_{0,\bar g}
  &= \alpha\int_0^T \int_M g(u,u)\vol(g)\,dt\;,\\
  \label{H1norm}
 \|\bar\nabla\bar u\|^2_{0,\bar g}&=\alpha \int_0^T\int_M
 \bar g^1_1(\bar \nabla \bar u,\bar \nabla \bar u)\vol(g)\,dt\;.
\end{align*}

To summarize, we propose to solve the optical flow problem on a moving
manifold by minimizing the energy functional
\begin{equation}\label{eq:Edef}
\boxed{
\begin{aligned}
\mathcal{E}(\bar{u}) &:= \mathcal{S}(\bar{u}) + 
\mathcal{R}(\bar{u}) \\
&=
\lVert \partial_t I+g(\nabla^g I,u)\rVert_{0,\bar g}^2
+ \beta\lVert \bar{u}\rVert^2_{0,\bar g}
 + \gamma \|\bar\nabla\bar u\|^2_{0,\bar g}\,,
\end{aligned}
}
\end{equation}
which is defined on
\[
\dom(\mathcal E)=\left\{\bar u \in H^1(\bar{M},T\bar{M}) :
  \bar u=(0\partial_t, u)\right\}\,.
\]
\begin{remark}
  Note that the regularization term depends implicitly on the
  parameter $\alpha$ as well. However, a large value of $\alpha$ leads
  to less time regularity, which is in contrast to the influence of
  the parameters $\beta$ and $\gamma$. Formally, the limit
  $\alpha\to\infty$ corresponds to no time regularization at all.
\end{remark}
\begin{remark}\label{rem:kirlansch}
	We stress the difference between the regularization term proposed in this article and the one from \cite{KirLanSch14a}, which is given by
	\begin{equation*}
		\mathcal R(u) = \int_0^T \int_{\mathcal{M}_t} \lambda_0 \abs{\proj_{T\mathcal{M}_t} \partial_t(Tf.u)}^2 + \lambda_1 \norm{\proj_{T^{1,1}\mathcal{M}_t} \nabla^{\R^3}(Tf.u)}^2.
	\end{equation*}
	Even though the latter functional is also a natural generalization of \cite{WeiSchn01} --- from an embedded point of view ---, there is no obvious metric on $\bar M$ for which it is a weighted homogeneous $H^1$-norm.
\end{remark}
\begin{remark}
  In the case $\beta = 0$,
  where $\mathcal{R}$ is the homogeneous Sobolev semi-norm,
  only variations of the movement on the manifold are penalized but not
  the overall speed of the movement. 
  In contrast, a positive value of $\beta$ encourages a low speed,
  which may lead to a systematic underestimation of the magnitude of the 
  computed flow. For this reason the choice $\beta = 0$ is usually preferable.
  Note, however, that one of the basic assumptions in
  our model is that most of the movement of the image is driven by the
  movement of the manifold. Thus, using a positive value of $\beta$
  can be justified and is somehow natural provided that this assumption holds.
  In addition, the actual numerical computation of the flow field is
  easier for $\beta > 0$ because the condition of the resulting linear
  equation becomes better with increasing $\beta$.
  Still, we have used the parameter choice $\beta = 0$
  for our numerical experiments later in the paper.
\end{remark}
\begin{remark}\label{rem:zeropartialt}
  It is also possible to identify the non-autonomous vector field
  $u$ on $M$ with the vector field $\hat{u}(t,x) := (1\partial_t,
  u(t,x))$ on $\bar{M}$, which incorporates the movement of the image
  both in time and space. If one does so, however, one has to be
  careful about the regularity term. Simply using the squared
  (weighted) $H^1$-norm of $\hat{u}$ has the
  undesirable effect that the natural movement of the manifold, which
  is given by the vector field $\hat{u}_0 := (1\partial_t,0)$, need
  not be of minimal energy for the regularization term: the vector
  field $\hat{u}_0$ is in general not covariantly constant. Instead of
  the norm of the vector field $\hat{u}$ itself one should therefore
  penalize the norm of the difference between $\hat{u}$ and
  $\hat{u}_0$. Doing so, one arrives at the same regularization
  term~\eqref{H1norm} as above, although the interpretation is
  slightly different.
\end{remark}

\subsection{Wellposedness}

The proof of the wellposedness of our model, that is, the question
whether the proposed energy functional $\mathcal{E}$
attains a unique minimizer in $\dom(\mathcal{E})$,
is quite straightforward.
In the following result, we denote by $W^{1,\infty}(\bar{M})$
the space of functions on $\bar{M}$ with an essentially bounded
weak derivative.

\begin{theorem}
  Assume that $\alpha, \beta, \gamma > 0$ and that $I \in
  W^{1,\infty}(\bar{M})$. Then the
  functional 
  \[
  \mathcal{E}(\bar{u}) = \mathcal{S}(\bar{u}) +
  \mathcal{R}(\bar{u})
  \]
  defined in~\eqref{eq:Edef} has a unique minimizer in
  $\dom(\mathcal{E})$.
\end{theorem}
\begin{remark}
We will see in Sect.~\ref{optimality_cond} that this optimization problem is in a natural way connected to Neumann boundary conditions.
If we want to consider mixed boundary conditions instead---Dirichlet in space and Neumann in time---we have to restrict the domain of the energy functional to
\begin{align*}
 \dom_0(\mathcal{E}):=\{\bar u \in \dom(\mathcal{E}): \bar u=0 \text{ on } [0,T]\times \partial M\}.
\end{align*}
The wellposedness result remains valid on $\dom_0(\mathcal{E})$.
\end{remark}

\begin{proof}
  The condition $I \in W^{1,\infty}(\bar{M})$ guarantees that the
  similarity term $\mathcal{S}(\bar{u})$ is finite for every square
  integrable vector field on $\bar{M}$; in particular it is
  proper. From the condition $\beta > 0$ we obtain that the
  regularization term $\mathcal{R}$ and therefore also the energy
  functional $\mathcal{E}$ is coercive. Thus $\mathcal{E}$ is a proper
  and coercive, quadratic functional on the Hilbert space
  $\dom(\mathcal{E})$, which implies the existence of a unique
  minimizer (cf.~\cite[Section~3.4]{SchGraGroHalLen09} or~\cite{Schn91a}).
\end{proof}

\begin{remark}\label{re:coercivity}
  The condition $\beta > 0$ is not necessary if there is another way
  of guaranteeing the coercivity of the regularization term. 
  This is for instance possible, if there exists no non-trivial
  covariantly constant vector field of the form $\bar{u} =
  (0\partial_t,u)$ on $\bar{M}$. In that case, the homogeneous Sobolev
  semi-norm $\lVert \bar\nabla \bar u\rVert^2_{0,\bar g}$
  is in fact a norm on $\dom(\mathcal{E})$ that is equivalent to the
  standard Sobolev norm, and therefore also the parameter choice $\beta = 0$
  guarantees the coercivity of $\mathcal{E}$.
  Note that this condition is independent of the moving image $I$.

  More generally, even if there are non-trivial, admissible,
  covariantly constant vector fields on $\bar{M}$, the energy function
  will be still coercive for $\beta = 0$, as long as no such
  vector field satisfies the optical flow equation
  $\partial_t I + g(\nabla^g I,u) = 0$. Note,
  however, that the numerical computation of a minimizer may become
  difficult, because the problem, though still wellposed, may become
  ill-conditioned as the parameter $\beta$ approaches zero.
\end{remark}

\begin{remark}
  With a similar argumentation one can show that the functionals
  proposed in~\cite{KirLanSch13,KirLanSch14a} are wellposed provided they are coercive, compare Remark \ref{rem:kirlansch}.
  Because there all the regularization terms penalize only the
  derivative of the vector field $u$ but not its size, the coercivity
  will only hold if one of the conditions in Remark~\ref{re:coercivity}
  is satisfied.
\end{remark}

\subsection{The optimality conditions}\label{optimality_cond}

\begin{lemma}\label{th:gradient}
The $L^2$ gradient of the optical flow energy functional $\mathcal E$
is given by
\begin{align*}
 \operatorname{grad}\mathcal E(\bar u)= 2\Big(\partial_t I + g\big(\nabla^{g} I,u\big)\Big)(0,\nabla^{g} I)+ 2\beta \bar u+ 2\gamma \Delta^{\operatorname{B}} \bar u\;.
\end{align*}
\begin{itemize}
 \item For $\mathcal E$ seen as functional on $\operatorname{dom}(\mathcal E)$ its domain of definition is the set of all vector fields $\bar u \in \operatorname{dom}(\mathcal E)$ satisfying Neumann boundary conditions, i.e.,
$$\operatorname{dom}(\operatorname{grad}(\mathcal E))=\left\{\bar u\in \operatorname{dom}(\mathcal E): \bar \nabla_{\nu}\bar u \big|_{\partial\bar{M}}=0  \right\}\,,$$
where $\nu$ denotes the normal to the boundary of $\bar M$ with respect to $\bar{g}$.
\item For $\mathcal E$ restricted to $\operatorname{dom}_0(\mathcal E)$ its domain of definition is the set of all vector fields $\bar u \in \operatorname{dom}_0(\mathcal E)$ satisfying mixed boundary conditions, more precisely,
$$\operatorname{dom}_0(\operatorname{grad}(\mathcal E))=\left\{\bar u\in \operatorname{dom}_0(\mathcal E): \bar \nabla_{\nu}\bar u \big|_{\{0,T\}\times M}=0  \right\}\,.$$
Note here that on $\{0,T\}\times M$ the normal vector $\nu$ is given by $\nu=\partial_t$.
\end{itemize}

\end{lemma}
Here 
$\Delta^{\operatorname{B}}$ denotes the Bochner Laplacian of $\bar g$, which is defined via
$$\Delta^{\operatorname{B}}=\bar \nabla^* \bar \nabla\;,$$
with $\bar \nabla^*$ denoting the $L^2$-adjoint of the covariant derivative.
The Bochner Laplacian differs only by a sign from the usual Laplace Beltrami operator. 

\begin{proof}
We  calculate the gradients for the two terms separately.  Using a variation $\overline{\delta  u}=(0,\delta u)$ we obtain the following expression for the variation of the similarity term: 
\begin{align*}
&D\left(\mathcal{S}(\bar u)\right)(\overline{\delta  u})= 2\alpha\int_0^T \hspace{-.2cm}\int_{M} \Big(\partial_t I + g\big(\nabla^{g} I,u\big)\Big) g\big(\nabla^{g} I,\delta u\big) \operatorname{vol}(g) \,dt\;.
\end{align*}
From this equation one can easily read off the $L^2$ gradient of the similarity term. It reads as
\[\operatorname{grad}^{L^2}\left(\mathcal{S}(\bar u)\right)= 2\Big(\partial_t I + g\big(\nabla^{g} I,u\big)\Big)(0,\nabla^{g} I)\;.\] 
The variation of the regularization term is given by 
\begin{align*}
& D\left(\mathcal{R}(\bar u)\right)(\overline{\delta u}) =2\alpha\beta \int_0^T \int_{M} \bar g( \bar u, \overline{\delta u})\vol(g)\,dt+2\alpha\gamma \int_0^T\int_M \bar g^1_1(\bar \nabla \bar u,\bar \nabla \overline{\delta u})\vol(g)\,dt\\
 &\qquad=2\alpha \beta \int_0^T\int_M \bar g( \bar u, \overline{\delta u})\vol(g)\,dt+2\alpha\gamma \int_0^T\int_M \bar g(\bar \nabla^* \bar \nabla \bar u,\overline{\delta u})\vol(g)\,dt
 \\&\qquad\qquad+2\alpha\gamma \int_0^T \int_{\partial M} \bar g( \bar \nabla_{\nu} \bar u,\overline{\delta u})\vol(g)|_{\partial M} dt
 +2\alpha\gamma\int_M \bar g( \bar \nabla_{\partial_t} \bar u,\overline{\delta u})\vol(g)\Big|_0^T
 \;.
\end{align*}
The second step consists of a partial integration using  the $L^2$ adjoint of the covariant derivative, which we denote by $\bar \nabla^*$. The last two term in the above expression are the boundary term that results from the partial integration.
From this we can read off the formula for the gradient of the regularization term. Taking into account that the outer normal vector to the boundary of $\{0,T\}\times M$  is given by $\nu=\partial_t$ this concludes 
the proof on $\operatorname{dom}(\mathcal E)$. For $\operatorname{dom}_0(\mathcal E)$ the situation is simpler, since the first boundary integral is already zero if $\overline{\delta u} \in \operatorname{dom}_o(\mathcal{E})$.
\end{proof}

Because of the strict convexity of the energy functional $\mathcal{E}$,
a vector field $\bar{u}$ is a minimizer if and only if it is an element
of $\operatorname{dom}(\operatorname{grad}(\mathcal E))$ and $\operatorname{grad}\mathcal{E}(u) = 0$.
Thus we obtain the following result:

\begin{theorem}
  The minimizer of the  energy functional $\mathcal{E}$ on $\operatorname{dom}(\mathcal E)$ defined in~\eqref{eq:Edef}
  is the unique solution $\bar{u} = (0\partial_t,u)$ of the equation
  \[
  \begin{aligned}
  \Big(\partial_t I + g\big(\nabla^{g} I,u\big)\Big)(0,\nabla^{g} I)+ 
  \beta \bar u+ \gamma \Delta^{\operatorname{B}} \bar u &= 0 &&\text{ in } \bar{M},\\
  \bar \nabla_{\nu} \bar u=(0,\nabla_{\nu} u) &=0  &&\text{ in }  [0,T] \times \partial M \\
  \bar \nabla_{\partial_t}\bar u &=0  &&\text{ in } \{0,T\}\times M.
  \end{aligned}
  \]
  If we restrict the energy functional to  $\operatorname{dom}_0(\mathcal E)$ it is the 
  unique solution of
  \[
  \begin{aligned}
  \Big(\partial_t I + g\big(\nabla^{g} I,u\big)\Big)(0,\nabla^{g} I)+ 
  \beta \bar u+ \gamma \Delta^{\operatorname{B}} \bar u &= 0 &&\text{ in } \bar{M},\\
  \bar \nabla_{\partial_t}\bar u &=0  &&\text{ in } \{0,T\}\times M,\\
  \bar{u} &= 0 &&\text{ in } [0,T] \times \partial M.
  \end{aligned}
  \]
\end{theorem}

\section{The optimality conditions in local coordinates}\label{coordinates}

The aim of this section is to express the previously derived
optimality conditions in a local coordinate chart in order to obtain
an implementable version of the previous sections. To simplify the exposition, we will restrict ourselves to the case of mixed boundary conditions.
Note that this includes in particular the situation where $M$ is
a compact manifold without boundary.

Let $(V,v)$ be a local chart on $M$ with coordinate frame
$\partial_1$, $\partial_2$. In the following we will use the Einstein
summation convention in order to simplify the notation.

The main computational difficulty is the computation of the Bochner Laplacian
$\Delta^{\operatorname{B}}=\bar \nabla^* \bar \nabla$, as it involves the
adjoint of the covariant derivative. This is most easily done in an orthonormal
frame with respect to the metric $\bar{g}$.
We stress here that the natural frame $(\partial_t,\partial_1,\partial_2)$
is in general not orthonormal, because
$\bar{g}(\partial_1,\partial_2) = g(\partial_1,\partial_2) = \langle \partial_1 f,\partial_2 f\rangle_{\R^3}$
will be different from 0.
Note, however, that the construction of the metric implies that
$\bar{g}(\partial_t,\partial_i) = 0$ for $i = 1,2$.
We can therefore obtain an orthonormal frame by scaling the
vector $\partial_t$ to unit length and, for instance, applying the Gram--Schmidt
orthogonalization process to the (time and space dependent) vectors $\partial_1$, $\partial_2$.
Doing so, we obtain an orthonormal frame of the form
\[
\bar{X}_0 = (\frac1{\alpha}\partial_t,0),
\qquad \bar X_1=(0,X_1),
\qquad \bar X_2=(0, X_2)
\]
with space dependent vector fields $X_1(t,\cdot)$ and $X_2(t,\cdot)$
on $M\cap V$.
The (time and space dependent) coordinate change matrix between these two bases will 
be denoted by $\bar A$; we have
\begin{equation*}
\begin{pmatrix}  \bar X_0\\ \bar X_1\\\bar X_2 \end{pmatrix}
=\bar A\begin{pmatrix}  \partial_t\\ \partial_1\\ \partial_2 \end{pmatrix}=
\begin{pmatrix} \frac1{\alpha} &0 &0 \\ 0 & a^1_1& a^2_1\\  0 & a^1_2& a^2_2\end{pmatrix}
\begin{pmatrix}  \partial_t\\ \partial_1\\ \partial_2 \end{pmatrix}\,.
\end{equation*}
Note that the coefficient function $a_1^2$ will be the constant 0
if the Gram--Schmidt process is used for orthogonalization.

In the orthonormal frame $\{\bar{X}_i\}$, the norm of $\bar{\nabla}\bar{u}$ can
be written as
\[
\bar{g}^1_1(\bar{\nabla}\bar{u},\bar{\nabla}\bar{u})
= \sum_i \bar{g}(\bar{\nabla}_{\bar{X}_i}\bar{u},\bar{\nabla}_{\bar{X}_i}\bar{u}),
\]
where $\bar{\nabla}_{\bar{X}_i}\bar{u}$ denotes the covariant derivative
of the vector field $\bar{u}$ along $\bar{X}_i$. 
Writing $\bar{u}$ as $\bar{u}^j\bar{X}_j$, the covariant derivative
can be computed as
\begin{equation*} 
  \bar{\nabla}_{\bar{X}_i} \bar{u} = \left( \bar{X}_i \bar{u}^j + \bar{u}^k \bar{\omega}_{ik}^j \right) \bar{X}_j\,,
\end{equation*}
where $\bar{\omega}_{ik}^j$ are the \emph{connection coefficients}.
These are defined by the equations
\begin{equation*} 
  \bar{\omega}^m_{ik} \bar{X}_m = \bar{\nabla}_{\bar{X}_i} \bar{X}_k.
\end{equation*}

In order to actually compute the connection coefficients, we use the
fact that they are related to the Christoffel symbols $\bar{\Gamma}_{k\ell}^i$ of $\bar{g}$
with respect to $(\partial_t,\partial_1,\partial_2)$ via
\begin{equation*}
  \bar \omega^j_{ik} = \left( \bar a^\ell_i \partial_\ell \bar a^m_k 
    + \bar a^\ell_i \bar a^n_k \bar\Gamma^m_{\ell n} \right) \bar a^h_j \bar g_{mh}\,.
\end{equation*}
Finally, the Christoffel symbols are defined as
\[
\bar \Gamma_{k\ell}^i=\frac12 \bar g^{im}\big(\bar g_{mk,\ell}+\bar g_{m\ell,k}-\bar g_{k\ell,m}\big)\,,
\]
where $\{\bar g^{ik}\}$ denote the coefficients of the metric $\bar{g}$ with respect
to $(\partial_t,\partial_1,\partial_2)$, and $\{\bar g_{ik,\ell}\}$ denotes the
partial derivatives of the coefficients of the inverse metric $\bar{g}^{-1}$.

Applying these definitions to our special situation, we
arrive at the following explicit formula for the optimality conditions
in local coordinates:

\begin{theorem}[Optimality conditions in local coordinates]\label{th:opt_coord}
  In the chart $(V,v)$, the unique minimizer 
  \[
  \bar{u} = (0\partial_t,u)= \left( 0,u^i X_i \right)
  \]
  of the energy
  functional $\mathcal{E}$ defined in~\eqref{eq:Edef} 
  solves the equation
  \begin{align}
    \left( A^j + B^j_m u^m + C^{mj}_k \partial_m u^k + D^{\ell
        m} \partial_{\ell m} u^j \right)X_j &= 0, \quad \text{in }
     V\times(0,T),\label{eq:main_eq} \\
    \left( \frac{1}{\alpha} \partial_t u^j + u^k \bar{\omega}^j_{0k} \right)X_j &= 0, \quad \text{on } V\times\{0,T\},
    \label{eq:neuman_coord}
  \end{align}
  where
  \begin{align*}
    A^j	&= \partial_t I \partial_i I g^{ik} b^j_k, \\
    B^j_m &= \partial_\ell I a^\ell_m \partial_i I g^{ik} b^j_k + \beta \delta^j_m - \gamma \textstyle\sum_i( \bar a^\ell_i \partial_\ell \bar\omega^j_{im} + \bar\omega^n_{im} \bar\omega^j_{in} ) - \frac{\gamma}{\alpha} \bar\omega^j_{0m}\bar{\Gamma}^n_{n0} , \\
    C^{mj}_k &= \textstyle -\gamma \sum_i (\delta^j_k \bar a^\ell_i \partial_\ell \bar a^m_i + 2\bar\omega^j_{ik} \bar a^m_i) -\frac{\gamma}{\alpha} \delta^j_k \bar a^m_0 \bar{\Gamma}^n_{n0}  ,\\
    D^{\ell m} &= \textstyle -\gamma \sum_i \bar a^\ell_i \bar a^m_i,
  \end{align*}
  and $b^j_k$ is the $X_j$-component of $\partial_k$.

  The connection coefficients $\bar{\omega}_{ik}^j$ of $\{X_i\}$ are given by
  \begin{equation}\label{eq:connection}
  \bar{\omega}^j_{ik} = \begin{cases}
    0, & \text{ if } j = 0 \text{ and either } i = 0 \text{ or } k = 0,\\
    \alpha a_i^\ell a_i^n \bar{\Gamma}_{\ell n}^0, & \text{ if } j = 0 \text{ and } i, k \neq 0,\\
    0, & \text{ if } j\neq 0 \text{ and } i=k = 0,\\
    \frac{1}{\alpha}\bigl(\partial_t a_k^m + a_k^n \bar{\Gamma}_{0n}^m\bigr)a_j^h g_{mh}, & \text{ if } i = 0 \text{ and } j,k\neq 0,\\
    \frac{1}{\alpha}a_i^\ell a_j^h g_{mh}\bar{\Gamma}_{\ell 0}^m, & \text{ if } k = 0 \text{ and } j,i \neq 0,\\
    \left( a^\ell_i \partial_\ell a^m_k + a^\ell_i a^n_k \Gamma^m_{\ell n} \right) a^h_j g_{mh}
    & \text{ if } j,i,k \neq 0.
  \end{cases}
  \end{equation}

  Moreover, the Christoffel symbols $\bar{\Gamma}_{ik}^j$ have the form
  \begin{equation}\label{eq:Christoffel}
  \bar{\Gamma}_{ik}^j = \begin{cases}
    0, & \text{ if } j = 0 \text{ and either } i = 0 \text{ or } k = 0,\\
    -\dfrac{\partial_t g_{ik}}{2\alpha^2}, & \text{ if } j = 0 \text{ and } i, k \neq 0,\\
    0, & \text{ if } j\neq 0 \text{ and } i=k = 0,\\
    g^{j\ell} \partial_t g_{k\ell}, & \text{ if } i = 0 \text{ and } j,k\neq 0,\\
    g^{j\ell} \partial_t g_{i\ell}, & \text{ if } k = 0 \text{ and } j,i \neq 0,\\
    2g^{j\ell} \langle \partial_{ik}f, \partial_\ell f\rangle_{\R^3}, & \text{ if } j,i,k \neq 0.
  \end{cases}
  \end{equation}
\end{theorem}

The derivations of this equation and of the coordinate expression of the connection symbols 
and the Christoffel symbols are postponed to the appendix.

%
%
\begin{figure}[h]
 \includegraphics[width=.19 \textwidth]{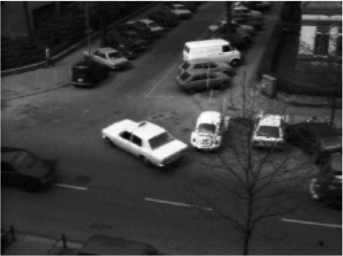}
 \includegraphics[width=.19 \textwidth]{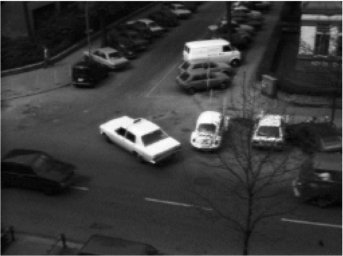}
 \includegraphics[width=.19 \textwidth]{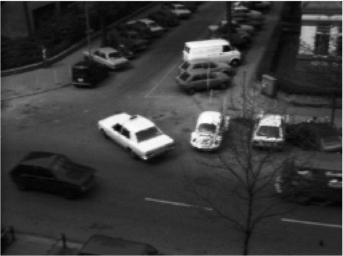}
 \includegraphics[width=.19 \textwidth]{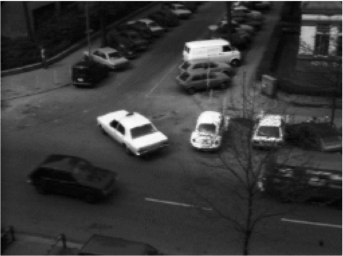}
 \includegraphics[width=.19 \textwidth]{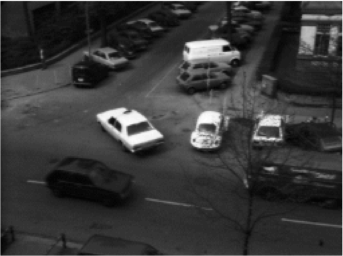}\vspace{.3cm}

 \includegraphics[width=.19 \textwidth]{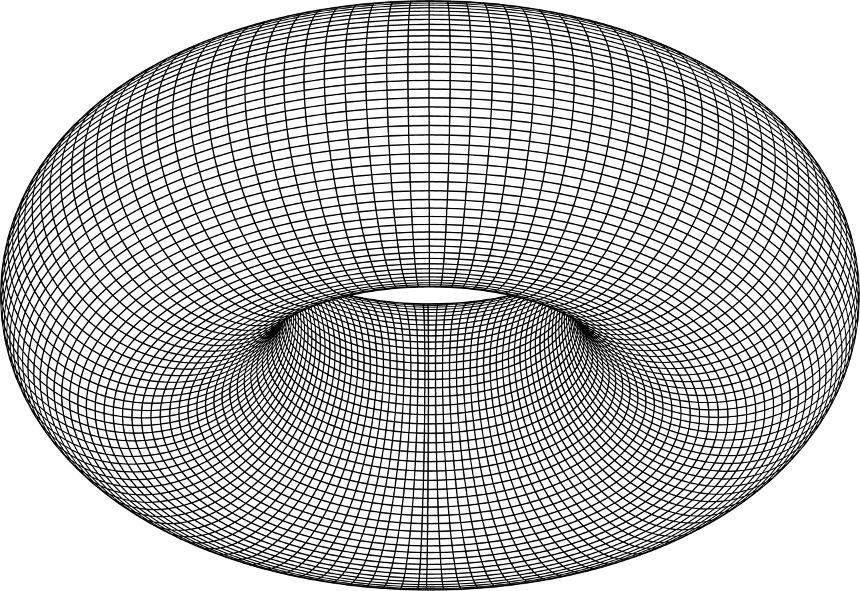}
 \includegraphics[width=.19 \textwidth]{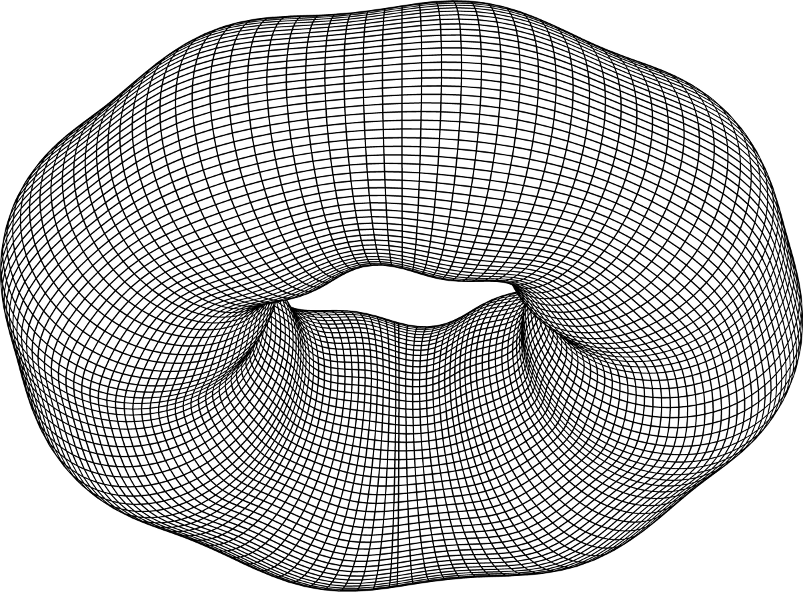}
 \includegraphics[width=.19 \textwidth]{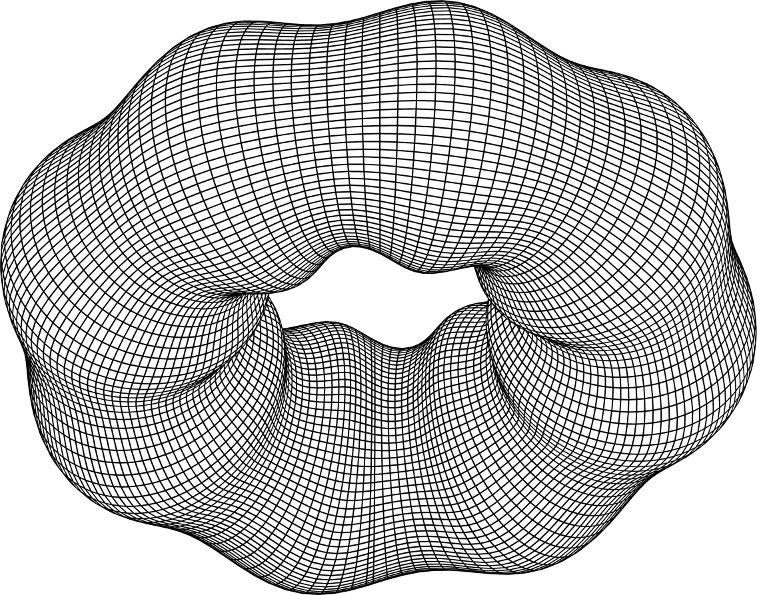}
 \includegraphics[width=.19 \textwidth]{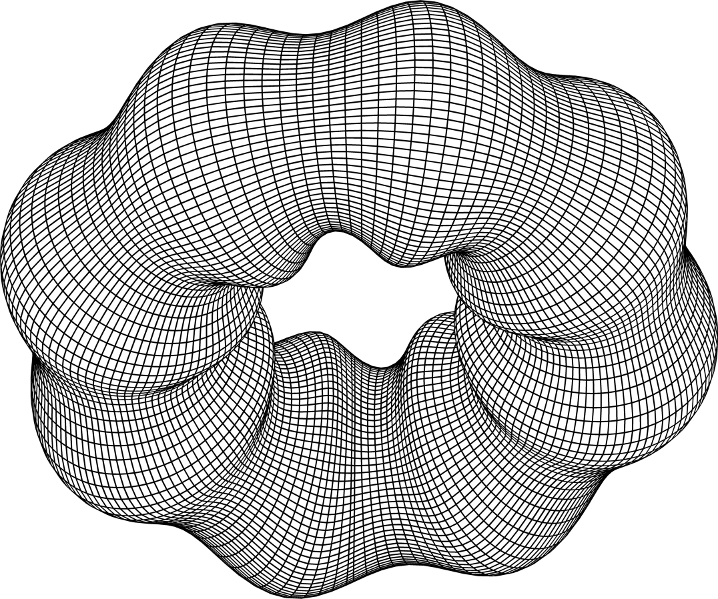}
 \includegraphics[width=.19 \textwidth]{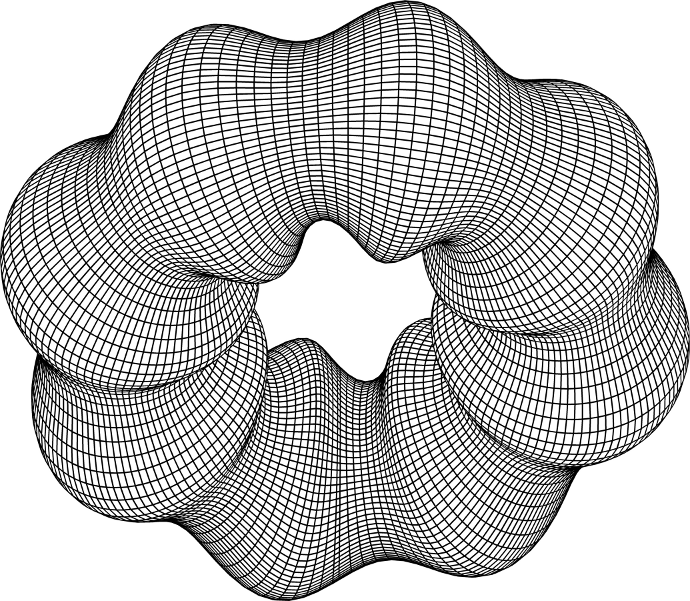}\vspace{.3cm}

 \includegraphics[width=.19 \textwidth]{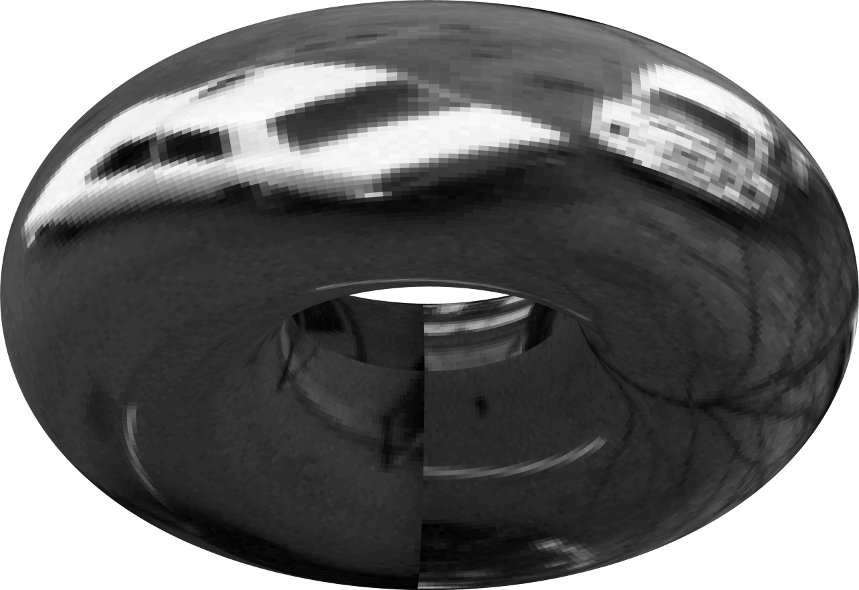}
 \includegraphics[width=.19 \textwidth]{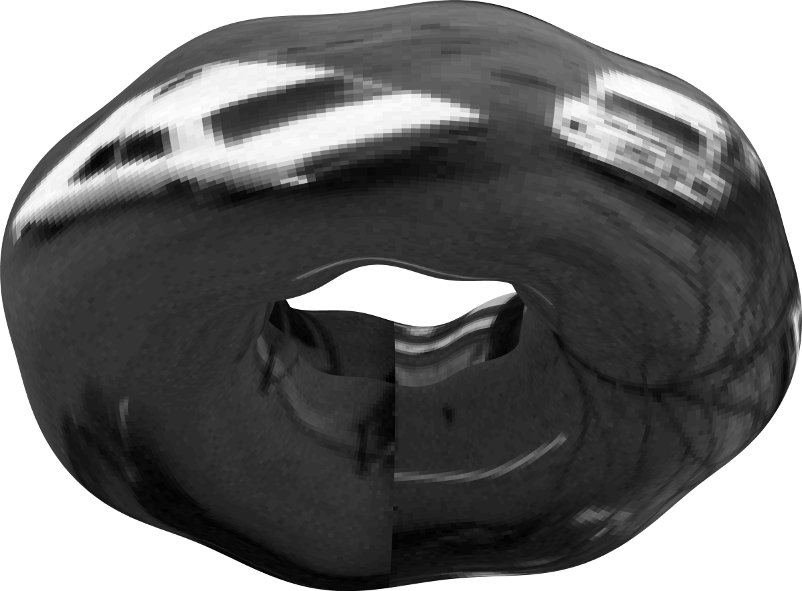}
 \includegraphics[width=.19 \textwidth]{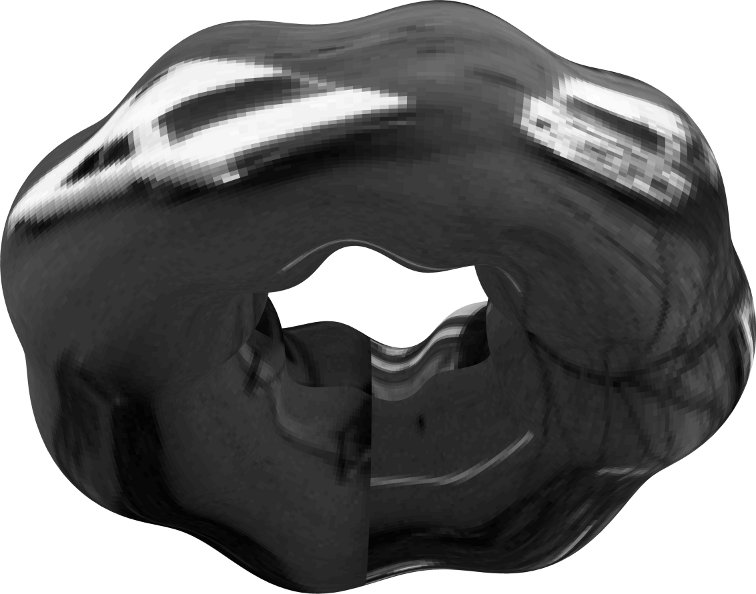}
 \includegraphics[width=.19 \textwidth]{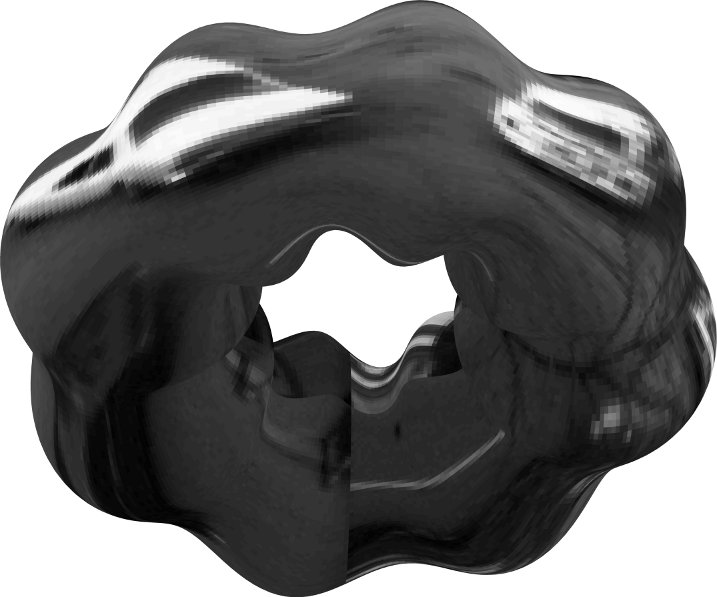}
 \includegraphics[width=.19 \textwidth]{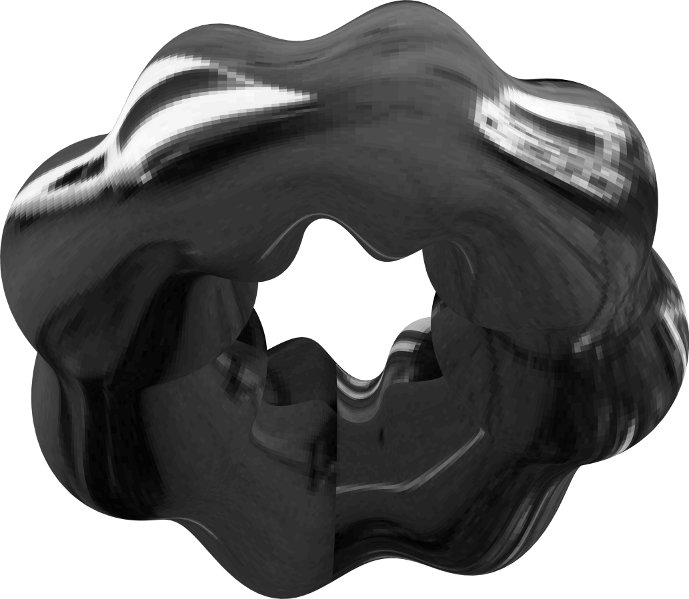}
 \caption{The data considered in Experiment I at frames 1, 6, 11, 16, and 20. Top row: the pulled back image sequence $I$. Middle row: the moving surface $\mathcal{M}_t$. Bottom row: the image sequence $\mathcal I$. }
\label{fig:e2data}
 \end{figure}

%
 
 \begin{figure}[ht]
 \includegraphics[width=.32 \textwidth]{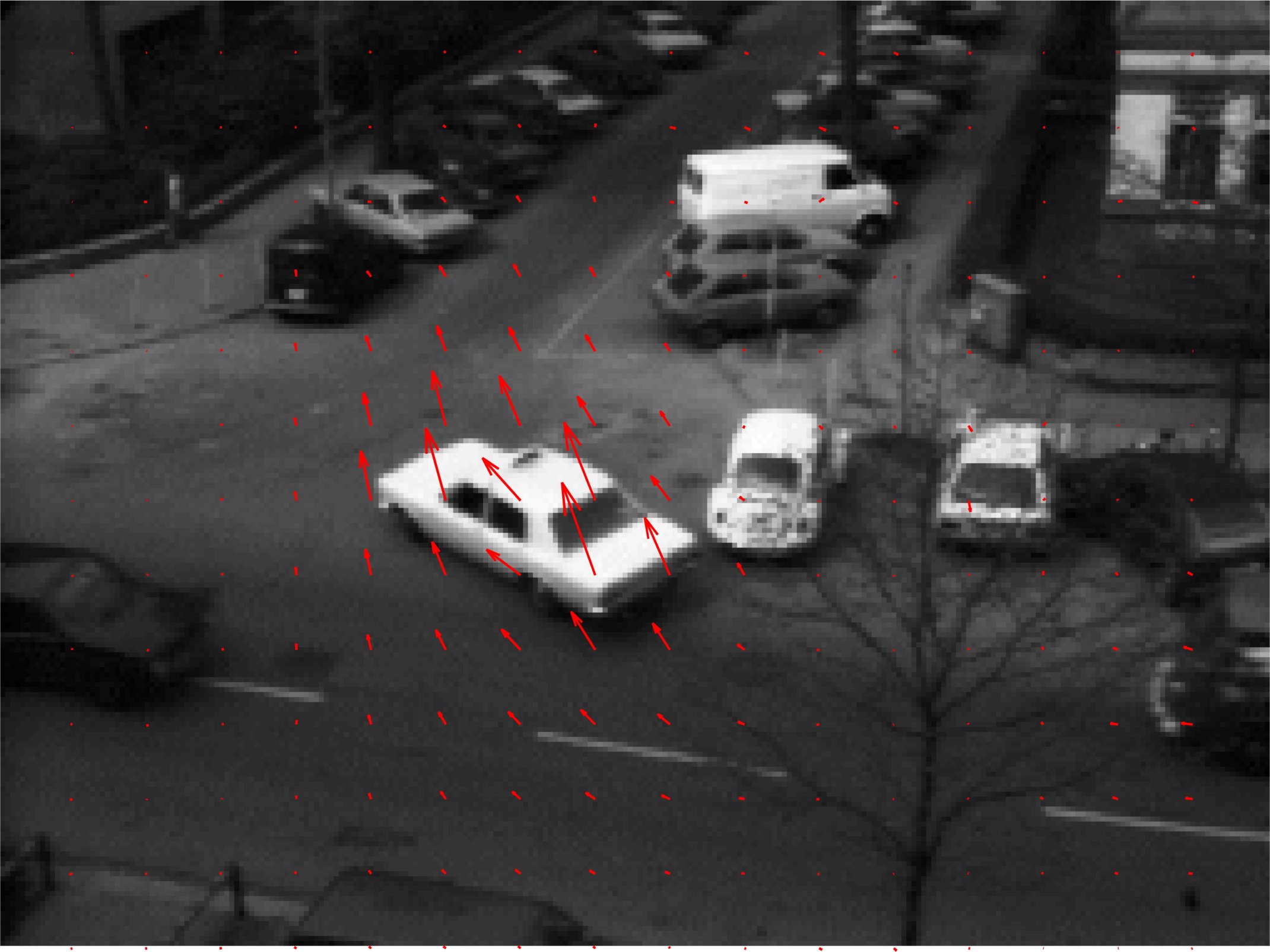}
 \includegraphics[width=.32 \textwidth]{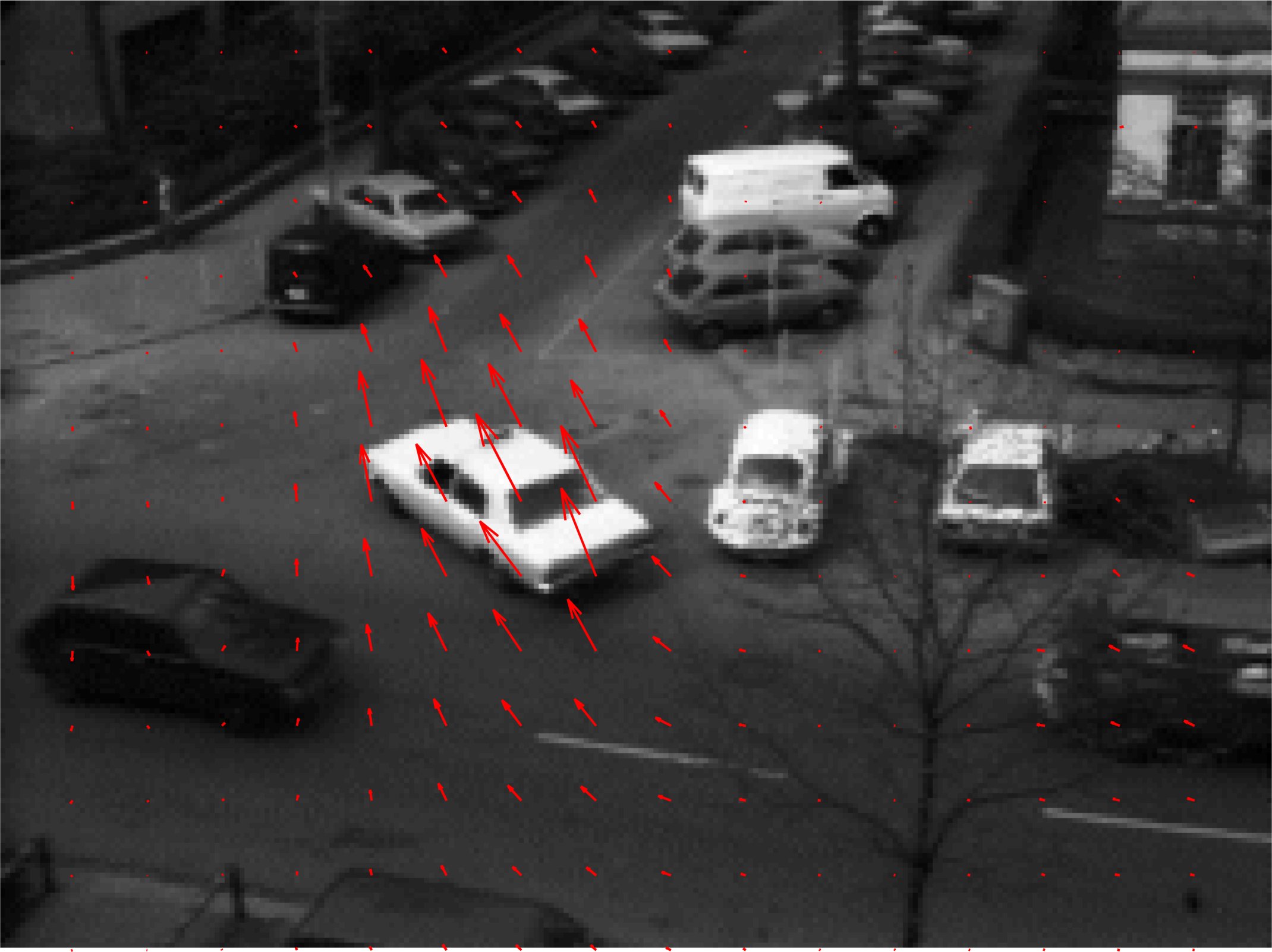}
 \includegraphics[width=.32 \textwidth]{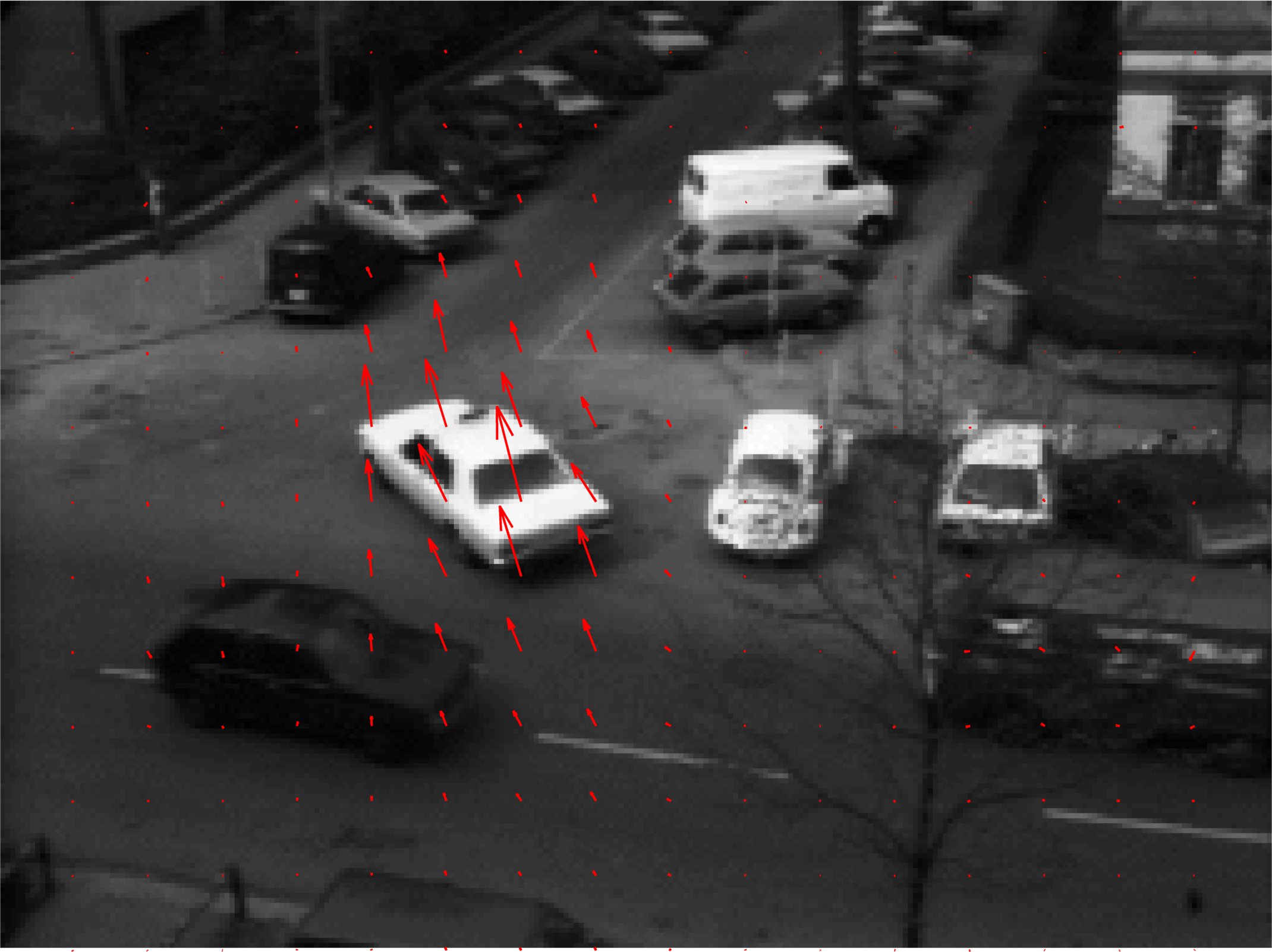}\vspace{.1cm}

 \includegraphics[width=.32 \textwidth]{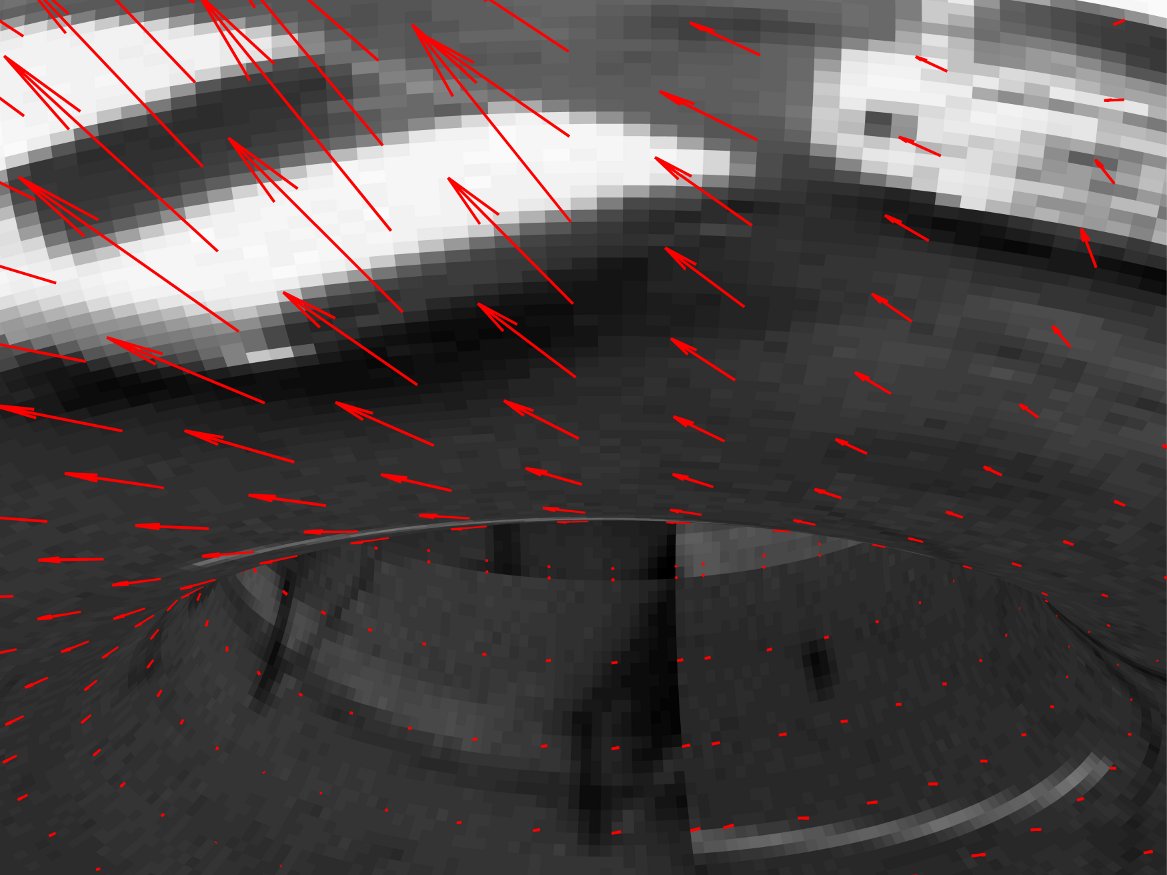}
 \includegraphics[width=.32 \textwidth]{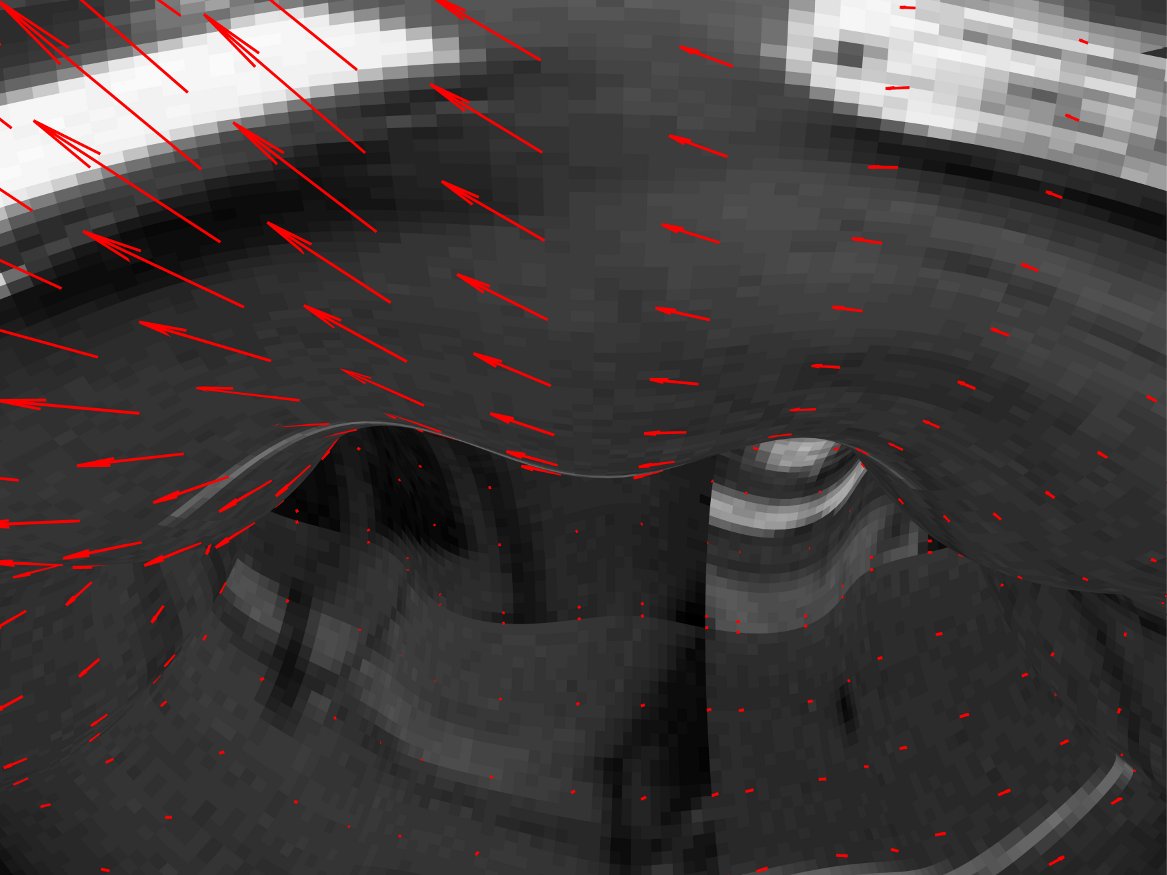}
 \includegraphics[width=.32 \textwidth]{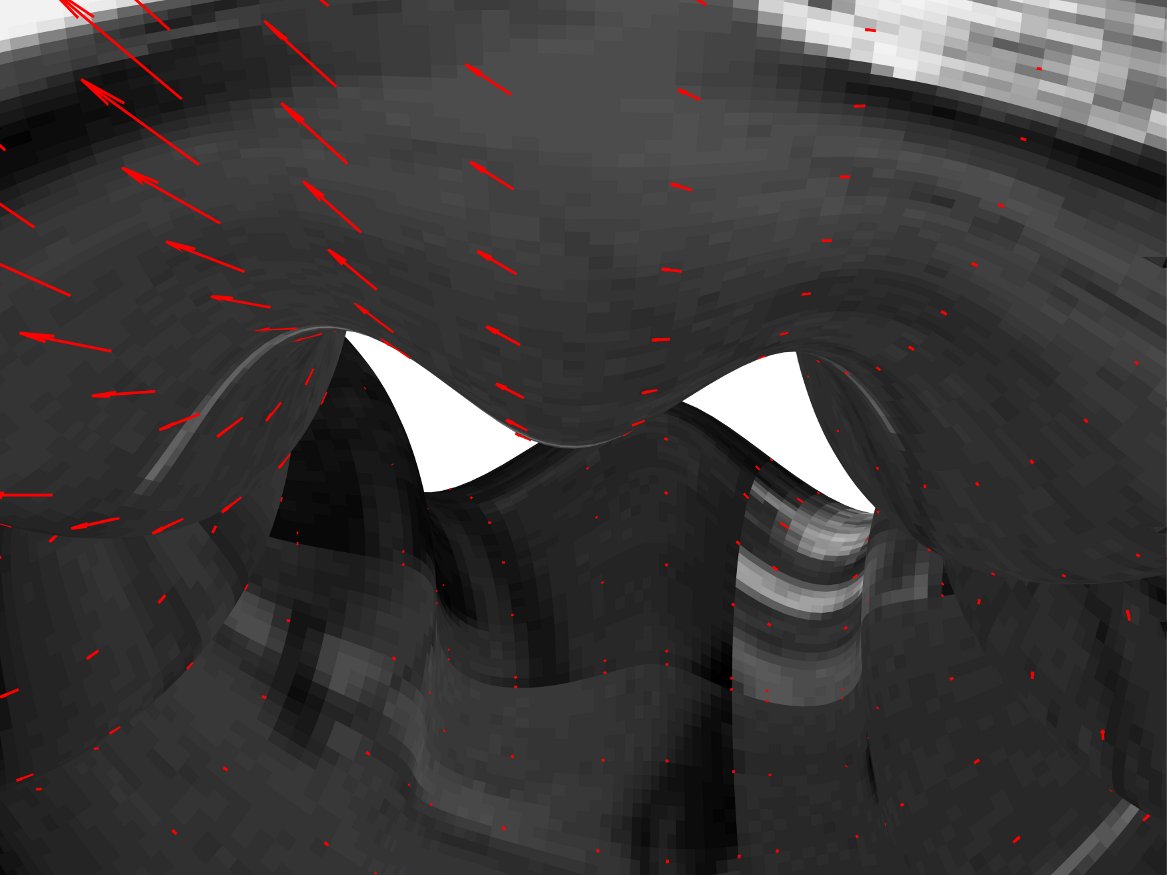}
 \caption{The optical flow vector field resulting from Experiment I at frames 1, 10, and 19. Top row: the pull back of the vector field. Bottom row: the vector field on the moving surface. The vectors have been scaled for better visibility.}
 \label{fig:e2arrows}
 \end{figure}

\section{Experimental Results}\label{experiments}
\paragraph{Numerical implementation.}
We illustrate the behaviour of the proposed model in two
experiments. In both of them the moving surface $\mathcal{M}_t \subset \mathbb{R}^3$ is parametrized globally by a function $f:[0,T]
\times M \to \mathbb{R}^3$, where $M\subset\mathbb{R}^2$ is a rectangular domain.
Note that, due to the chosen images, we can always set $\beta = 0$ and still have wellposedness, cf.~Remark~\ref{re:coercivity}.
We also conducted experiments with a positive value of $\beta$, which
yielded faster convergence of the numerical method. The main
difference to the results with $\beta=0$ were shortened flow fields.

We solve the optimality conditions from Thm.~\ref{th:opt_coord} with finite differences on a $k \times m \times n$ grid approximation of $\bar M$. Derivatives in all three directions are approximated by central differences and the resulting sparse linear system is solved with the standard Matlab implementation of the generalized minimal residual method (GMRES).

 
\begin{figure}
 \includegraphics[width=.24 \textwidth]{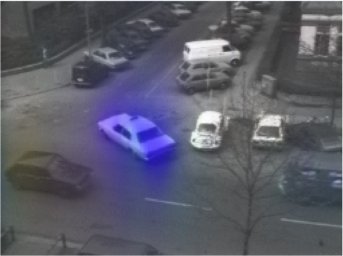}
 \includegraphics[width=.24 \textwidth]{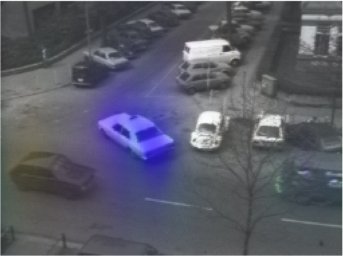}
 \includegraphics[width=.24 \textwidth]{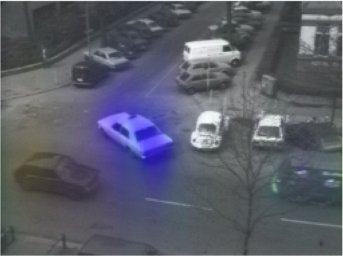}
 \includegraphics[width=.24 \textwidth]{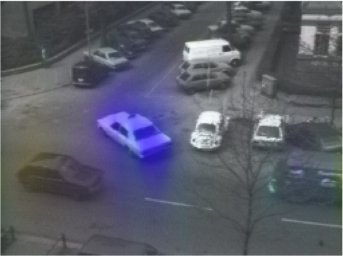}
\\
\includegraphics[width=.24 \textwidth]{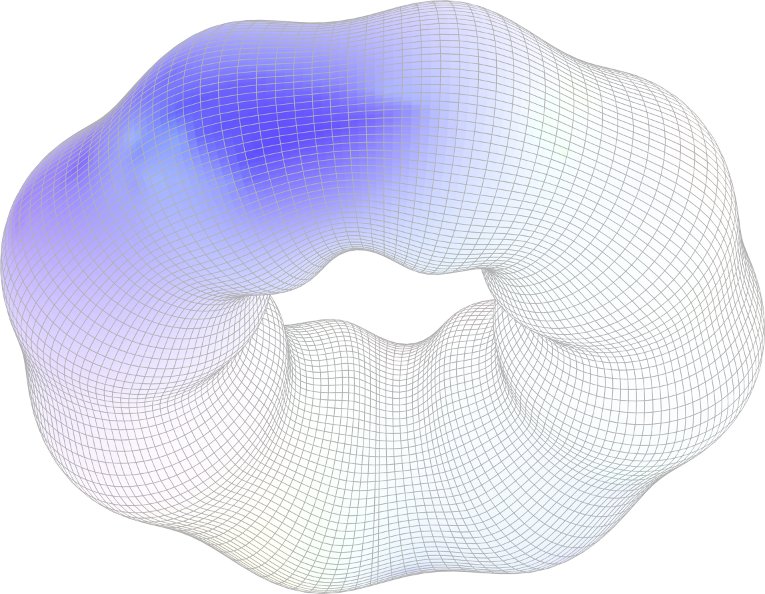}
\includegraphics[width=.24 \textwidth]{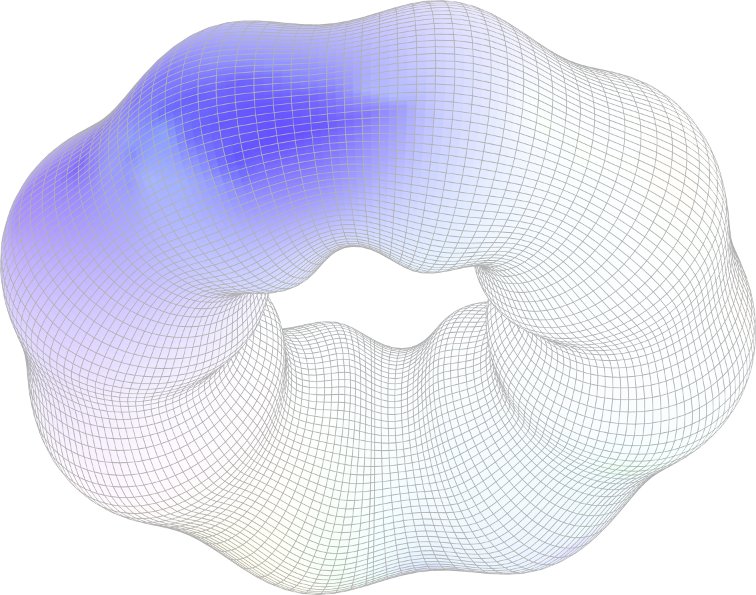}
\includegraphics[width=.24 \textwidth]{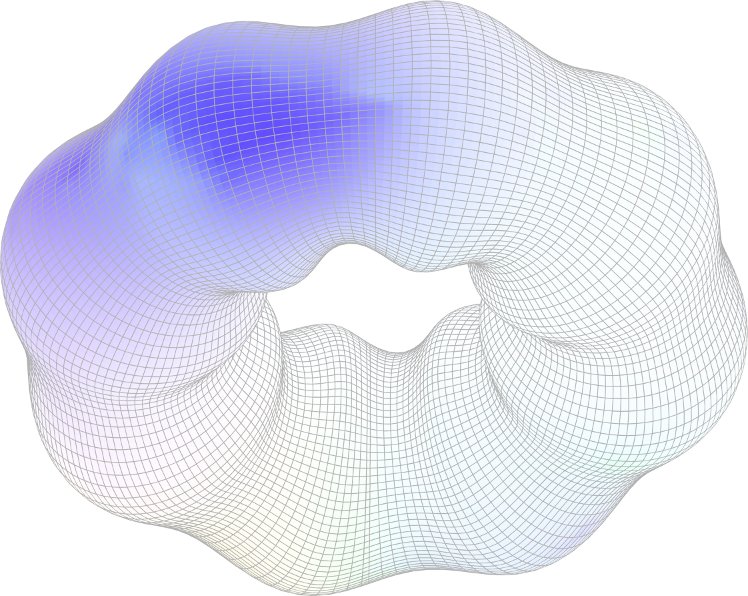}
\includegraphics[width=.24 \textwidth]{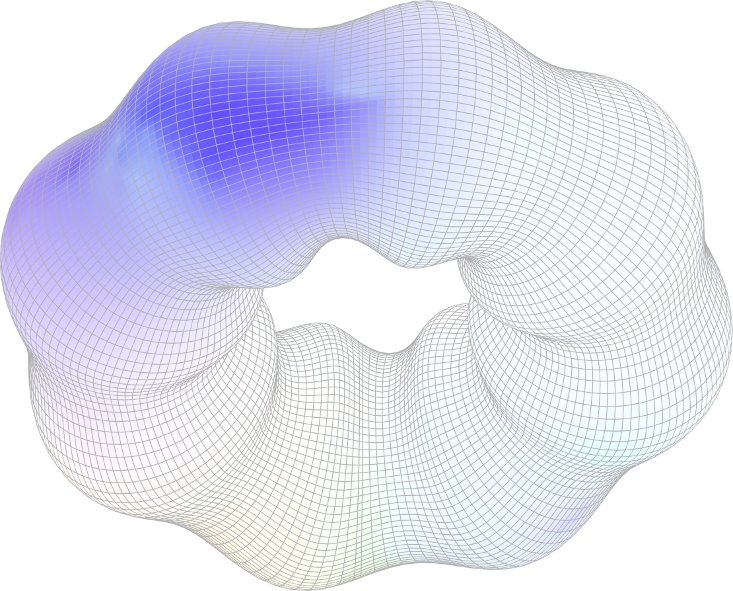}
\vspace{-0.2cm}
 \flushright{\includegraphics[width=.05 \textwidth]{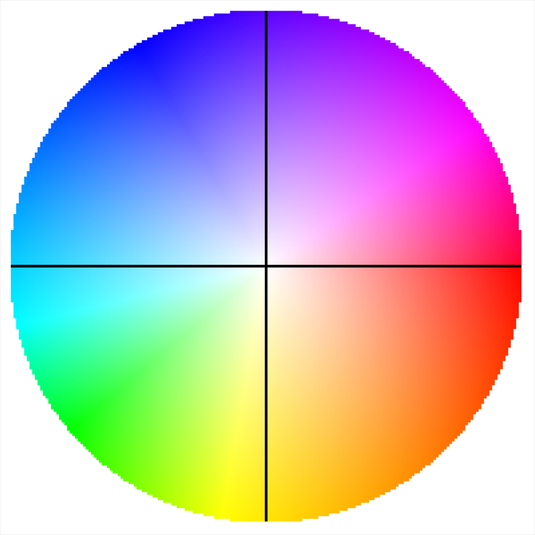} } 
  
 \caption{The color-coded optical flow vector field resulting from Experiment I at frames 10, 11, 12, and 13. First row: image sequence with pulled back vector field superimposed. Second row: vector field on the moving surface. The color wheel is shown at the very bottom.}
\label{fig:e2colors}
 \end{figure}

\begin{figure}
 \includegraphics[width=.24 \textwidth]{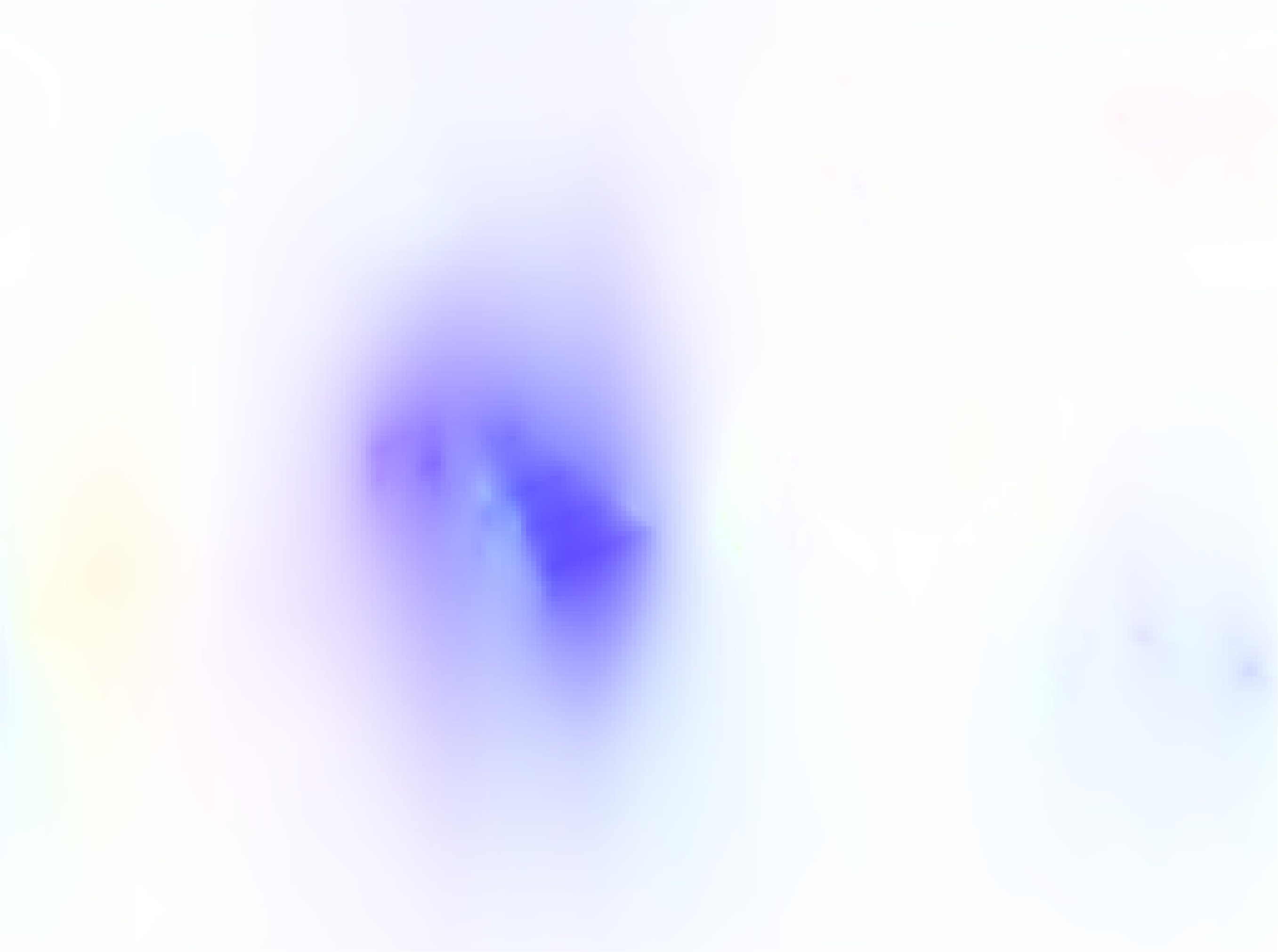}
 \includegraphics[width=.24 \textwidth]{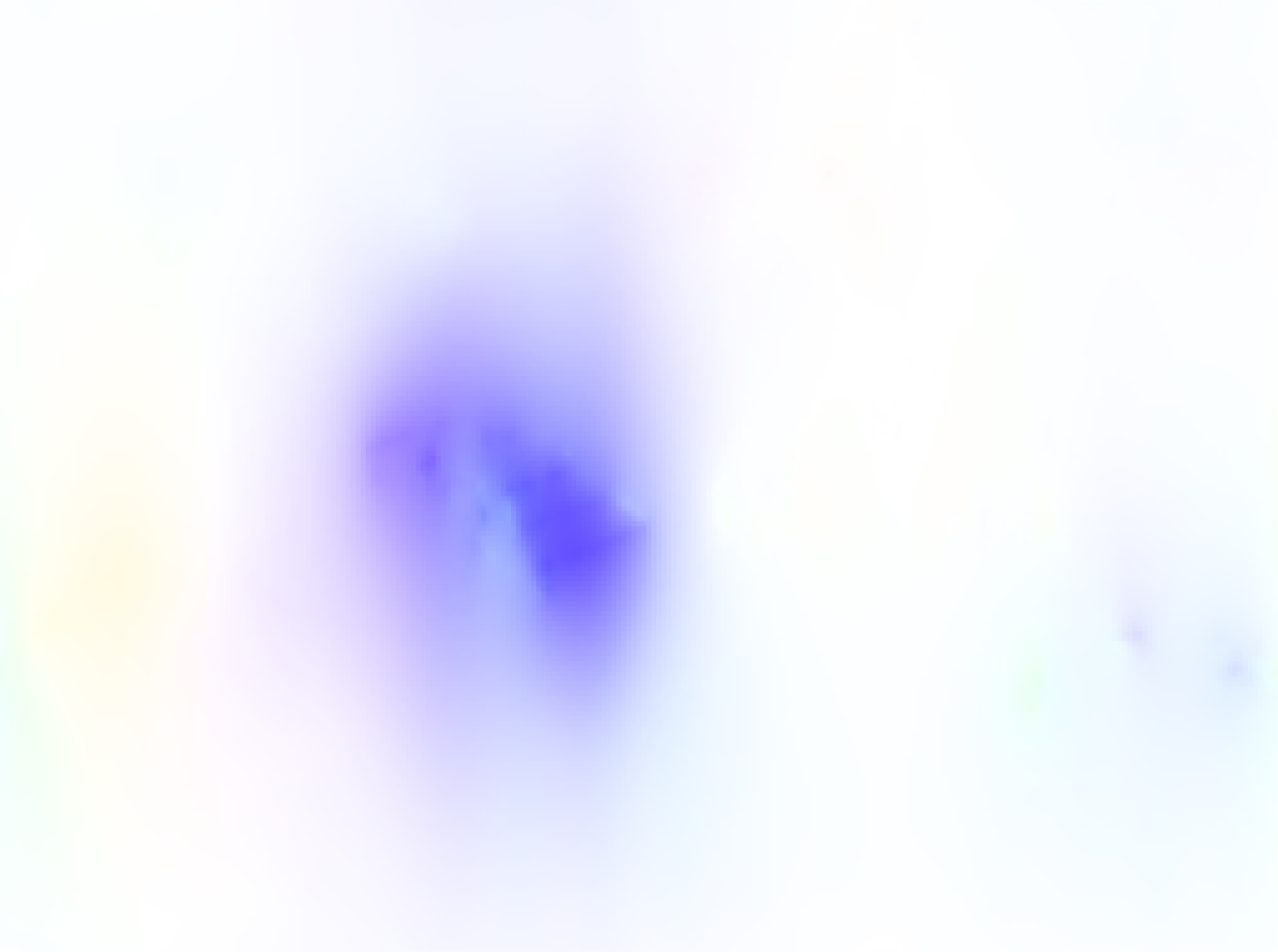}
 \includegraphics[width=.24 \textwidth]{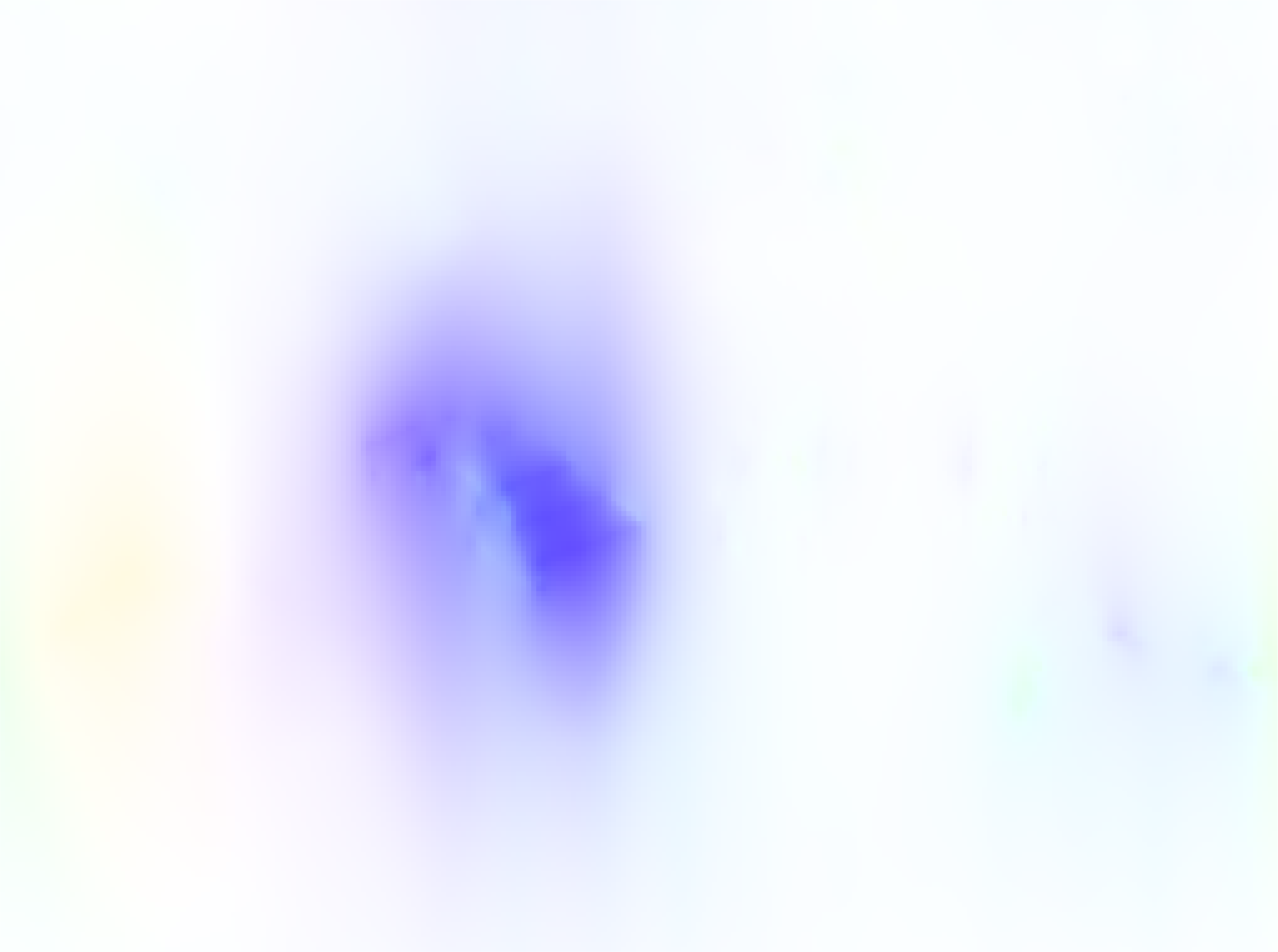}
 \includegraphics[width=.24 \textwidth]{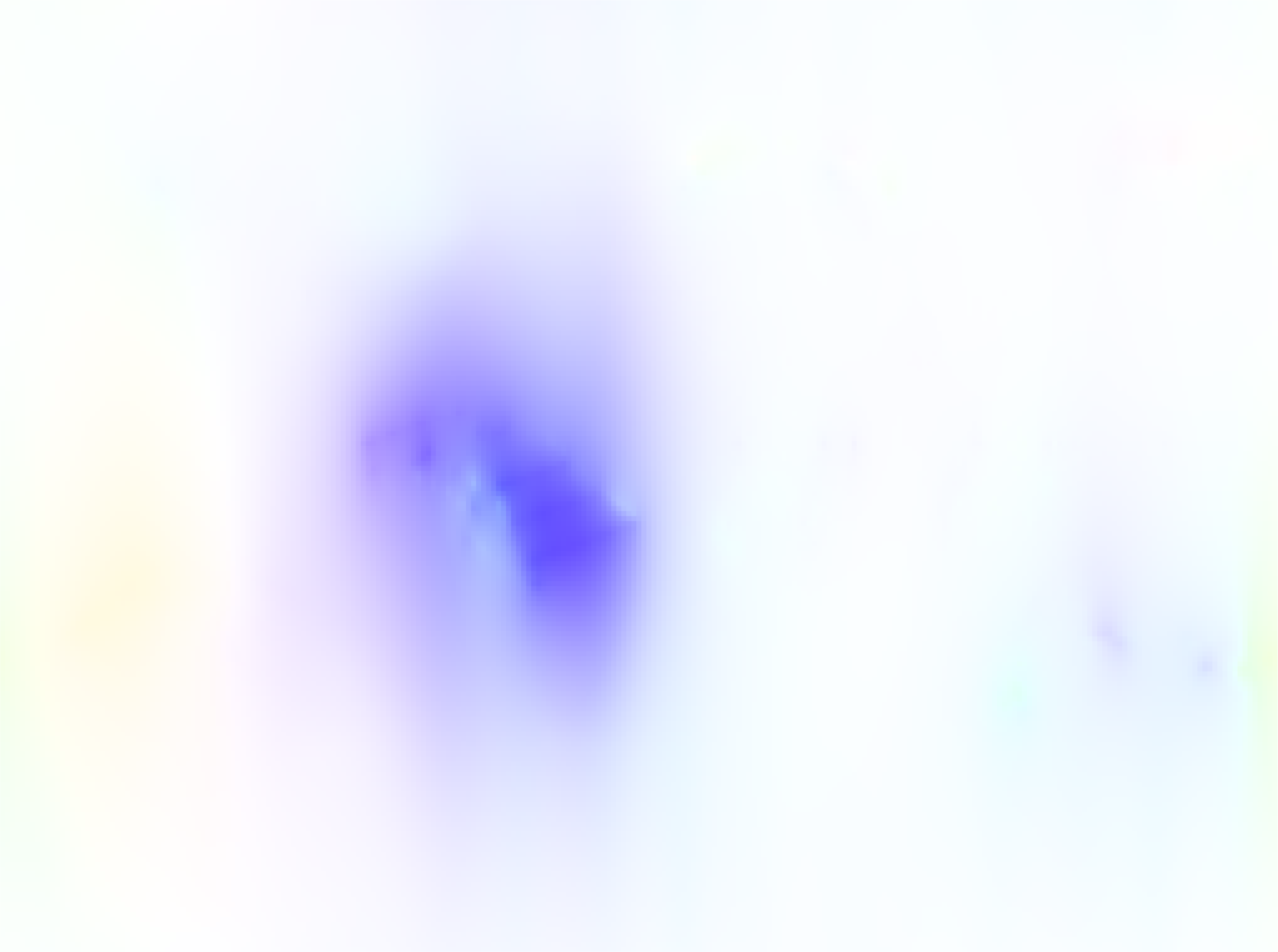}
\\
\includegraphics[width=.24 \textwidth]{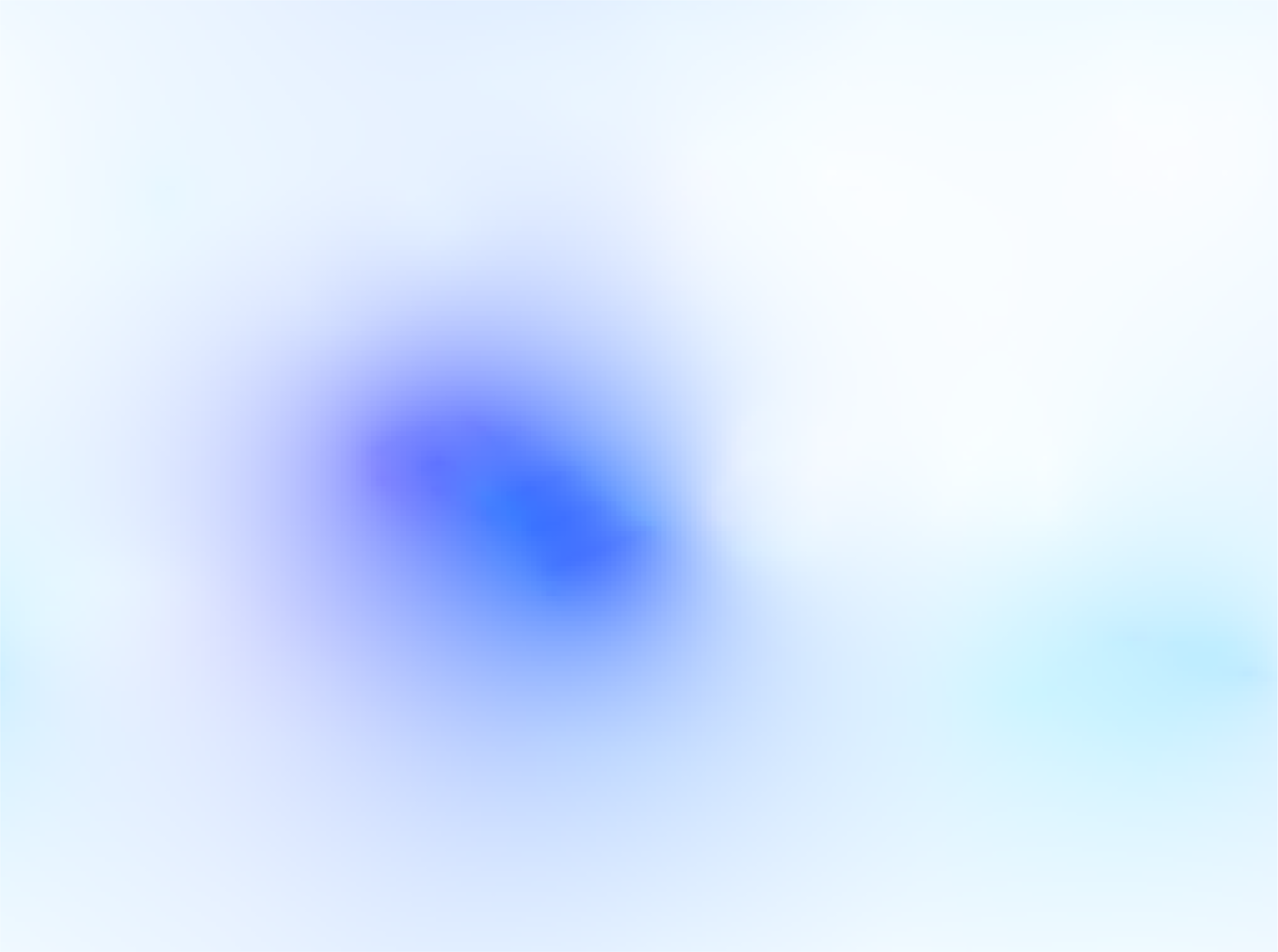}
\includegraphics[width=.24 \textwidth]{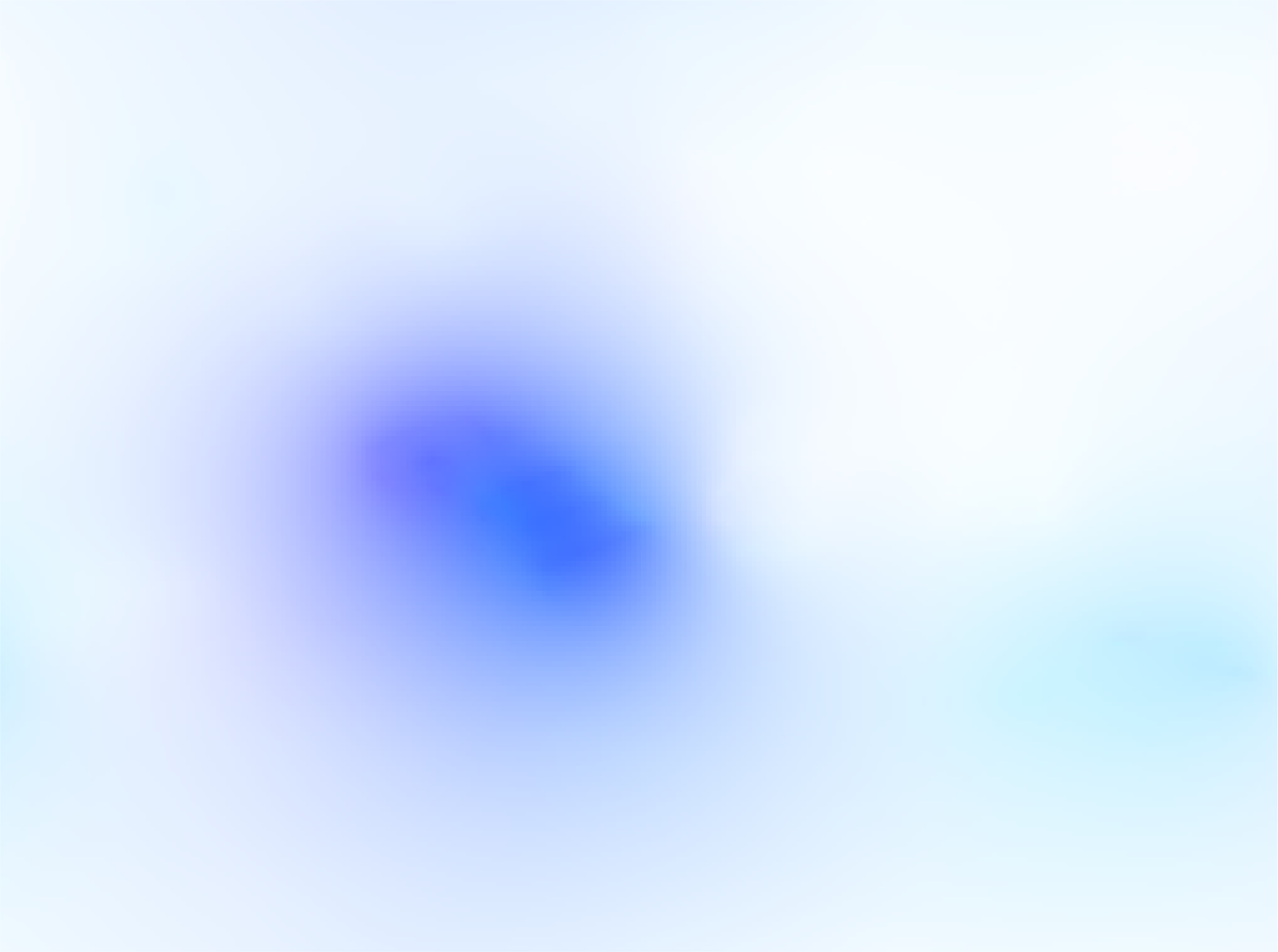}
\includegraphics[width=.24 \textwidth]{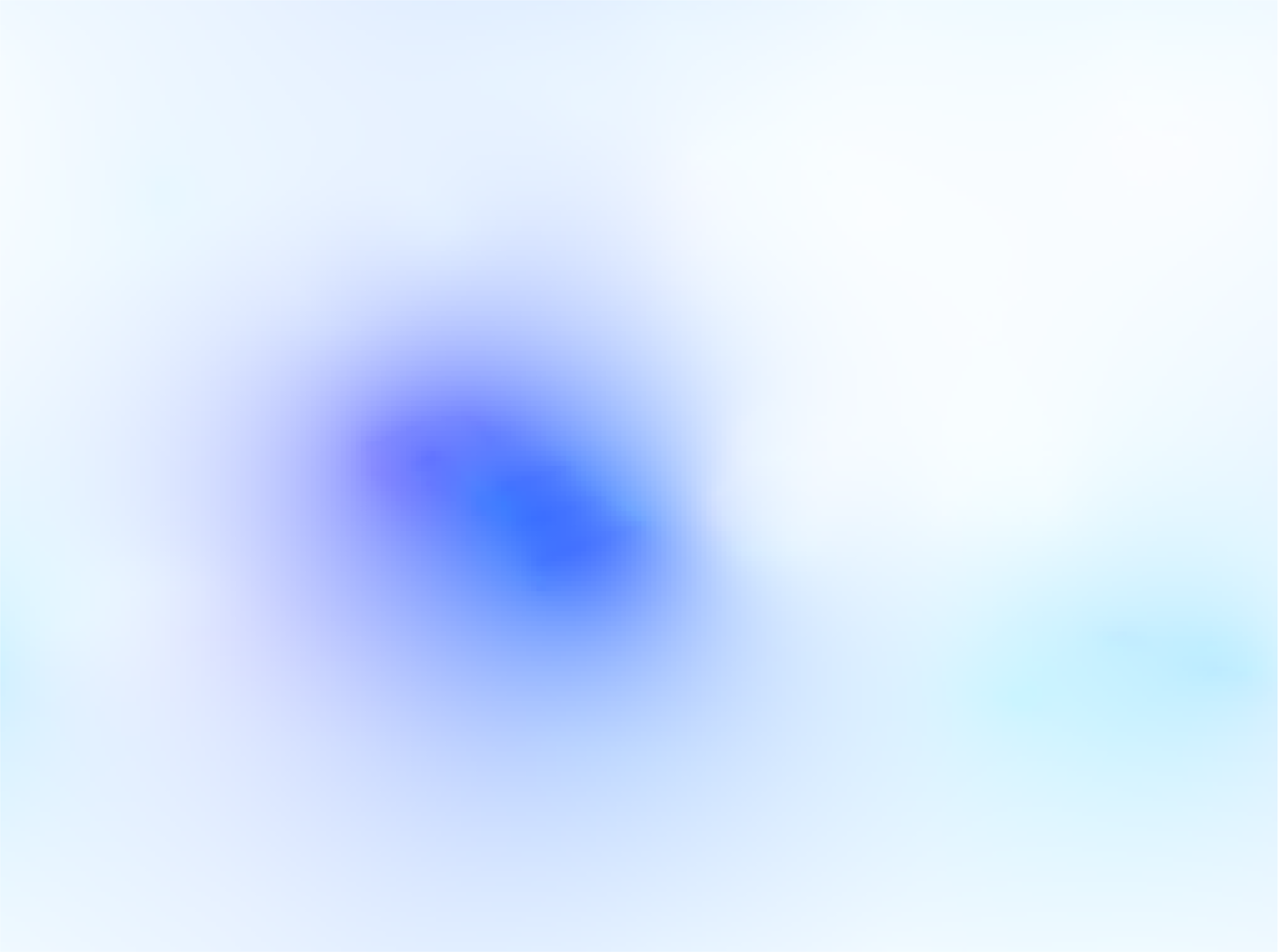}
\includegraphics[width=.24 \textwidth]{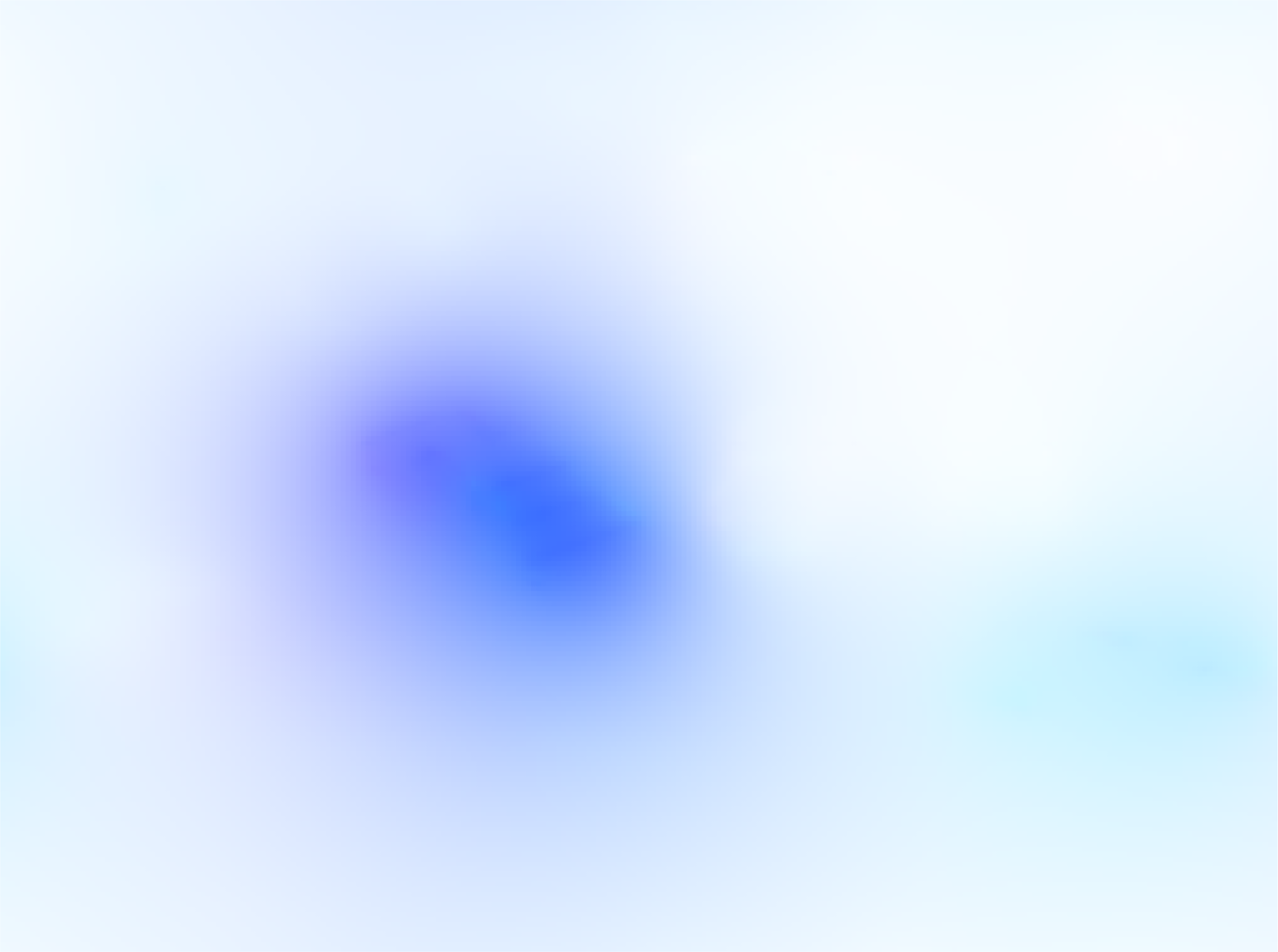}
\vspace{-0.2cm}
 \flushright{\includegraphics[width=.05 \textwidth]{figures/cw} }
 
 \caption{Comparison of the optical flow from Experiment I with the optical flow computed in the plane. First row: optical flow from Fig.~\ref{fig:e2colors}, first row, with removed image. Second row: optical flow computed in the plane with periodic boundaries. The color wheel is shown at the very bottom.}
\label{fig:e2comparison1}
\end{figure}

\paragraph{Experiment I, synthetic data.}
The first image sequence we apply our model to are $20$ frames of the well known Hamburg
Taxi sequence,\footnote{The movie can be dowloaded from
  \url{http://i21www.ira.uka.de/image_sequences/}.} scaled to the
unit interval. The sequence has a resolution of $255 \times 190$,
which leads to a total number of $9.7 \cdot 10^5$ grid points.

The surface we consider is a ring torus whose major circle turns into an ellipse while its tube of uniform thickness grows ripples over time. The corresponding embedding reads
\begin{equation*}
	f(t,x_1,x_2) =
		\begin{pmatrix}
			(R + \frac{t}{T} + r(t,x_1) \cos x_2) \cos x_1  \\
			(R + r(t,x_1) \cos x_2) \sin x_1  \\
			r(t,x_1) \sin x_2 
		\end{pmatrix},
\end{equation*}
where $R=2$, $r(t,x_1) = 1 + \frac{t}{5T} \sin 8x_1$, and $(x_1,x_2) \in [0,2\pi)^2$. In Fig.~\ref{fig:e2data} we show $I$, $\mathcal{M}_t$ and $\mathcal{I}$.

In Figs.~\ref{fig:e2arrows} and \ref{fig:e2colors} results for the parameter choice $\alpha=\gamma=1$ and $\beta=0$ are depicted. The finite difference step size $h$ was set to $1$ for all three directions. The GMRES algorithm was terminated after a maximum of $2000$ iterations with a restart every $30$ iterations. This led to a relative residual of $5.1 \cdot 10^{-3}$. In Fig.~\ref{fig:e2colors} we use the color coding from \cite{BakSchaLewRotBla11} to visualize the optical flow. This is done by applying it to the pulled back vector field first, and drawing the resulting color image onto $\mathcal{M}_t$ via $f$ afterwards.

Finally, we illustrate how the moving surface influences the optical flow vector field. To that end we repeat Experiment I on the flat torus with all parameters unchanged. That is, we compute the optical flow from the Hamburg Taxi sequence according to the model of Weickert and Schn\"orr \cite{WeiSchn01} only with periodic boundary conditions. In Fig.~\ref{fig:e2comparison1} we juxtapose the resulting vector field with  the optical flow computed on the deforming torus. A common measure for comparing two optical flow vector fields $u$ and $v$ is the angular error
\begin{equation*}
	\arccos \frac{\langle (1,u), (1,v) \rangle_{\mathbb{R}^3}}{|(1,u)|_{\mathbb{R}^3} |(1,v)|_{\mathbb{R}^3}}.
\end{equation*}
See \cite{BakSchaLewRotBla11} for example. The main purpose of adding the additional component $1$ to both vectors is to avoid division by zero. Extending the above definition in a straightforward way to vector fields in $\mathbb{R}^3$ we show in Fig.~\ref{fig:e2angerr} the angular error between the optical flow computed on the deforming and flat torus, respectively, both before and after pushforward to the deforming torus. Note that two unit vectors $u,v\in \mathbb{R}^2$ standing at an angle of $\pi/5$ would have an angular error of approximately $0.44$. Another common measure is the so-called endpoint error $|u-v|_{\mathbb{R}^3}$, which also takes into account the lengths of the vectors. However, since vector lengths are typically affected by the choice of regularization parameters and finding comparable values for the flat and deforming torus is not straightforward, we chose not to visualize the endpoint error.

\begin{figure}
	\includegraphics[width=.46 \textwidth]{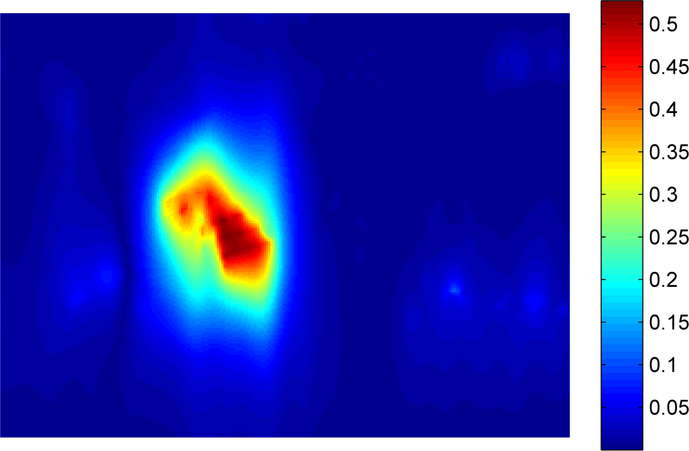}
	\hfill
	\includegraphics[width=.46\textwidth]{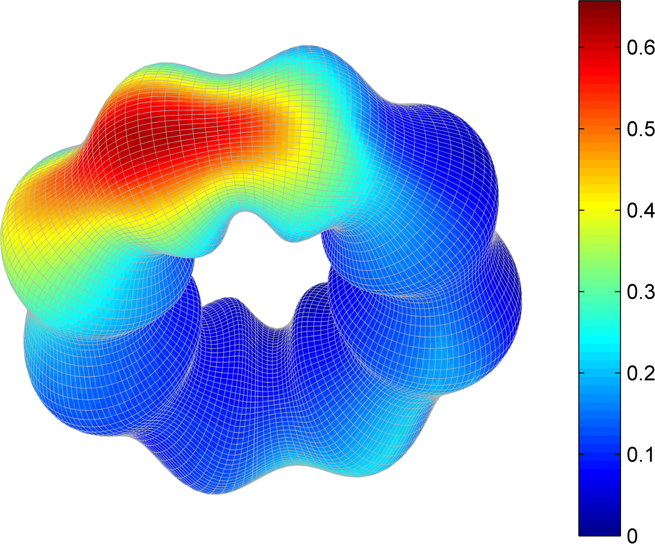}
	\caption{Angular errors at frame 19 between the optical flow vector fields computed on the flat and deforming torus, respectively. Left: $\mathbb{R}^2$ angular error between the pulled back vector fields. Right: $\mathbb{R}^3$ angular error between the pushed forward vector fields.}
	\label{fig:e2angerr}
\end{figure}

\begin{figure} 
 \includegraphics[width=.19 \textwidth]{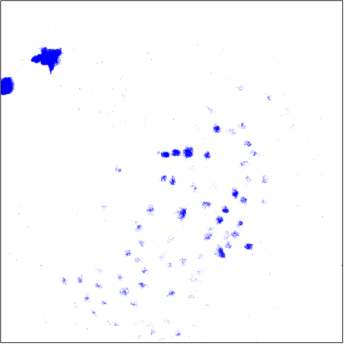}
 \includegraphics[width=.19 \textwidth]{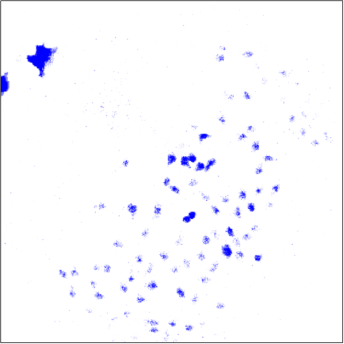}
 \includegraphics[width=.19 \textwidth]{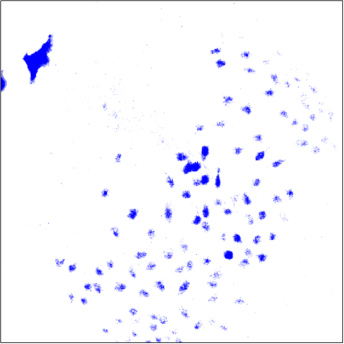}
 \includegraphics[width=.19 \textwidth]{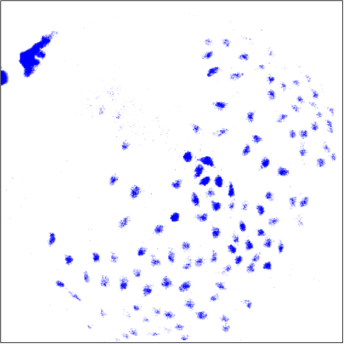}
 \includegraphics[width=.19 \textwidth]{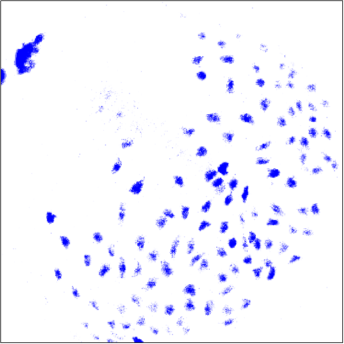}\vspace{.3cm}

 \includegraphics[width=.19 \textwidth]{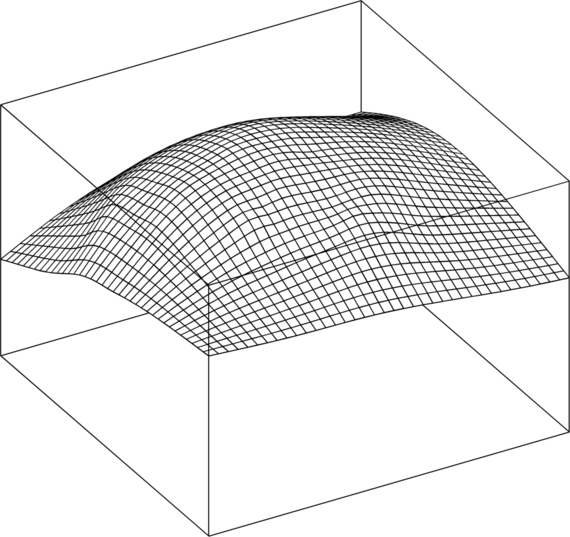}
 \includegraphics[width=.19 \textwidth]{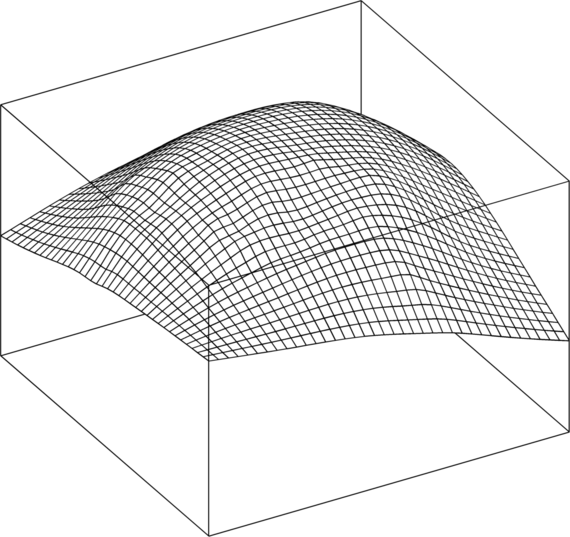}
 \includegraphics[width=.19 \textwidth]{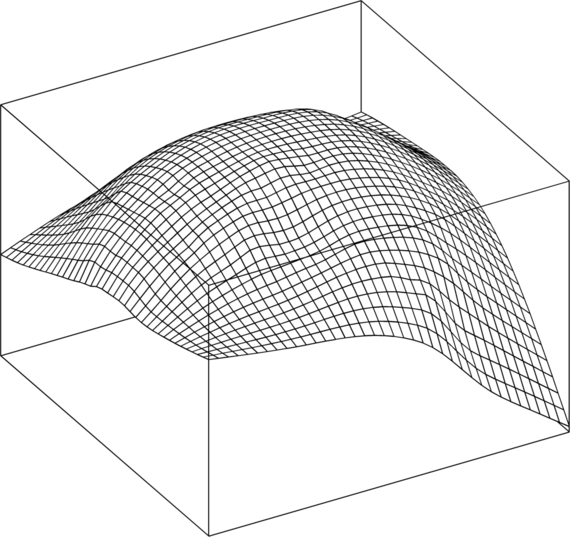}
 \includegraphics[width=.19 \textwidth]{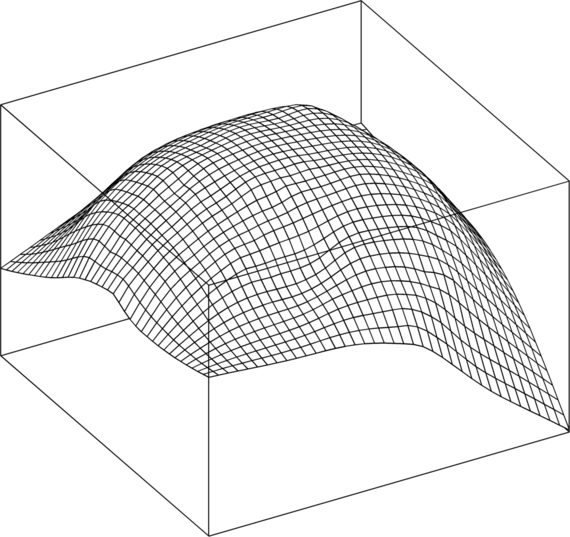}
 \includegraphics[width=.19 \textwidth]{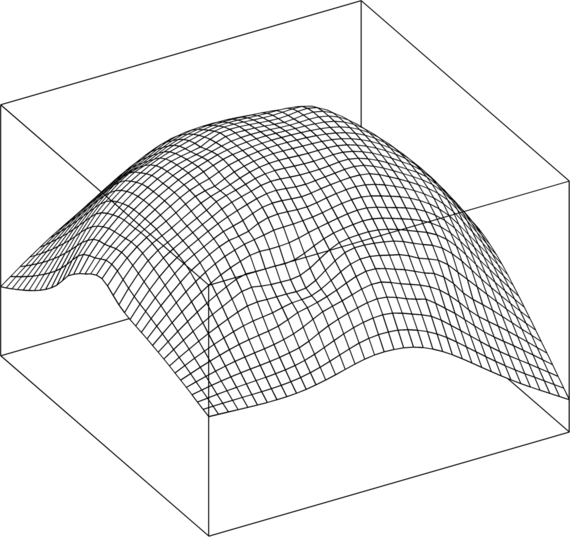}\vspace{.3cm}

 \includegraphics[width=.19 \textwidth]{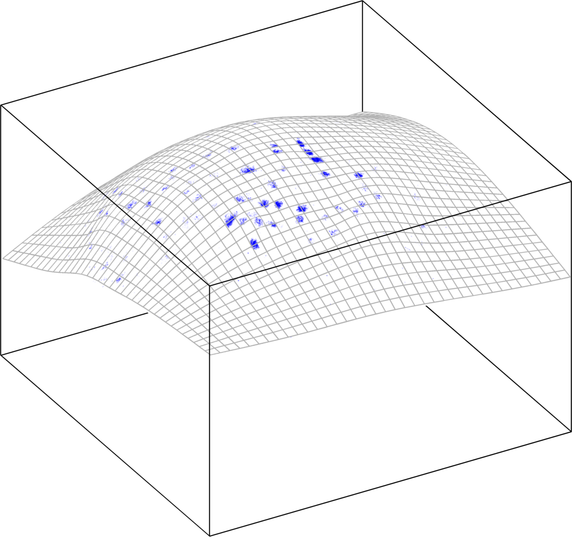}
 \includegraphics[width=.19 \textwidth]{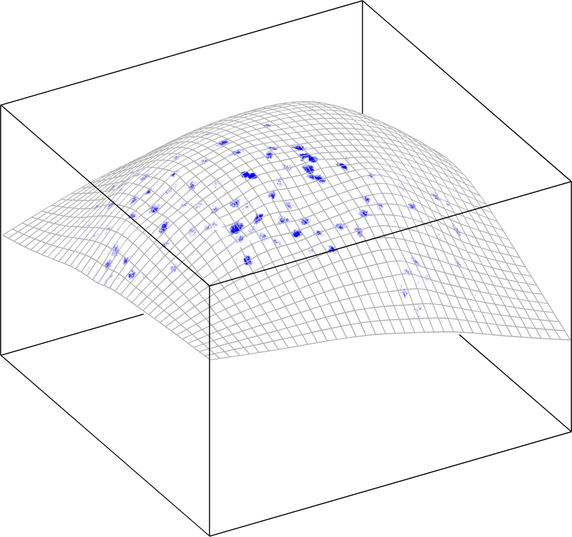}
 \includegraphics[width=.19 \textwidth]{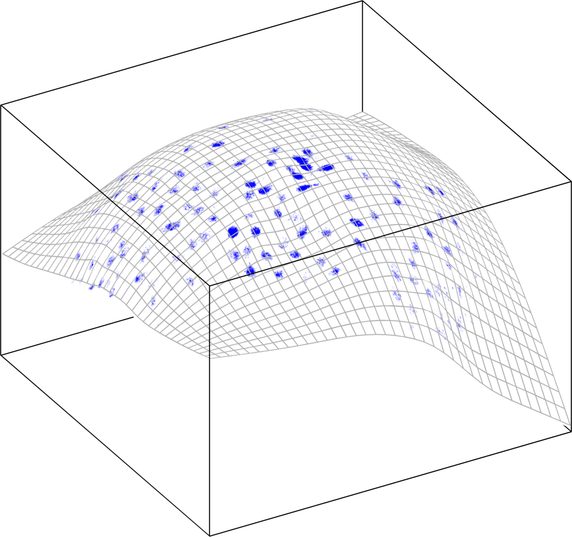}
 \includegraphics[width=.19 \textwidth]{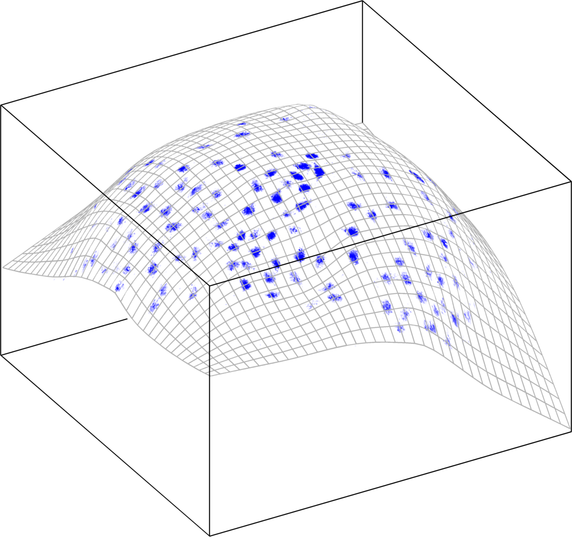}
 \includegraphics[width=.19 \textwidth]{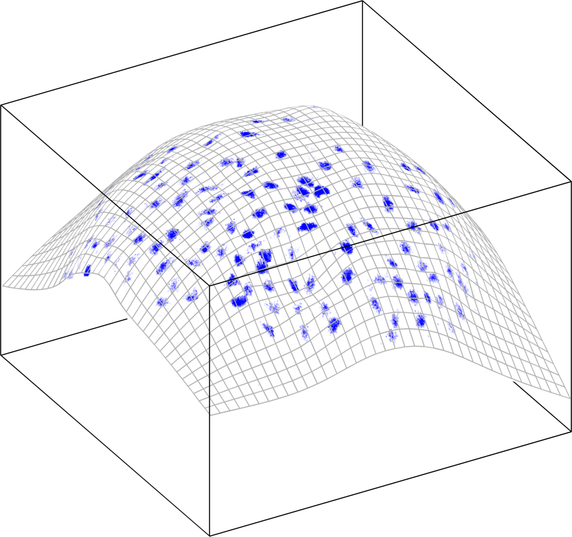}
 \caption{The data considered in Experiment II at frames 1, 6, 11, 16, and 20. Top row: the pulled back image sequence $I$. Middle row: the moving surface $\mathcal{M}_t$. Bottom row: the image sequence $\mathcal I$. }
\label{fig:e3data}
 \end{figure}
\begin{figure}
 \includegraphics[width=.24 \textwidth]{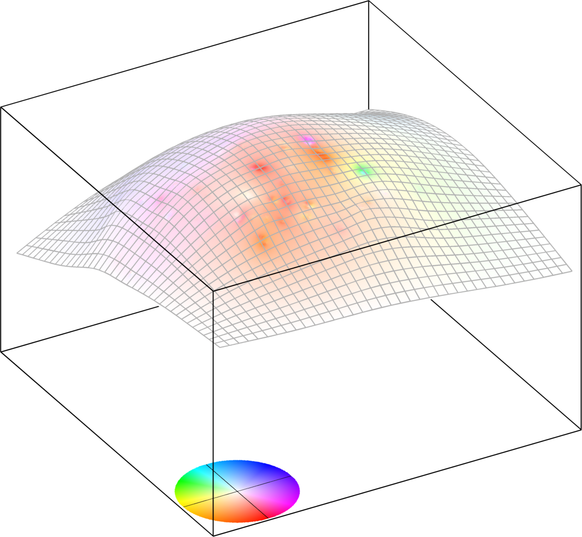}
 \includegraphics[width=.24 \textwidth]{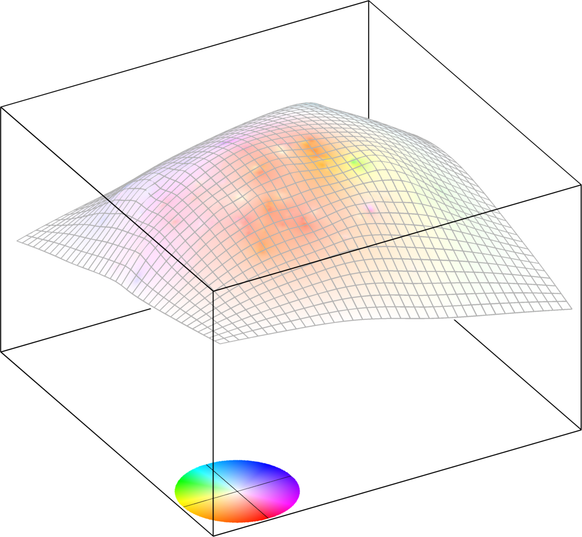}
 \includegraphics[width=.24 \textwidth]{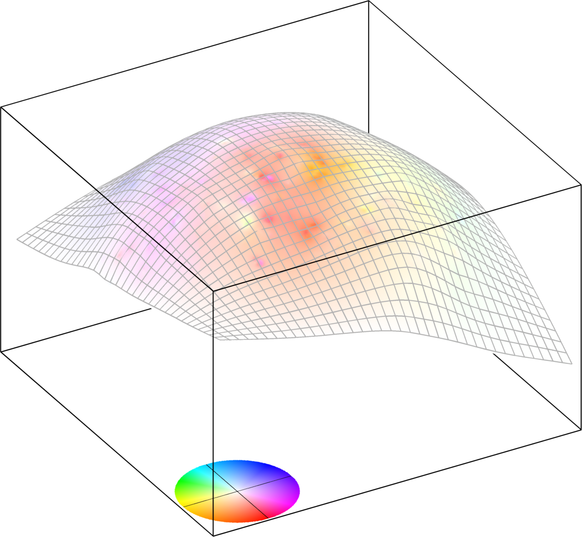}
 \includegraphics[width=.24 \textwidth]{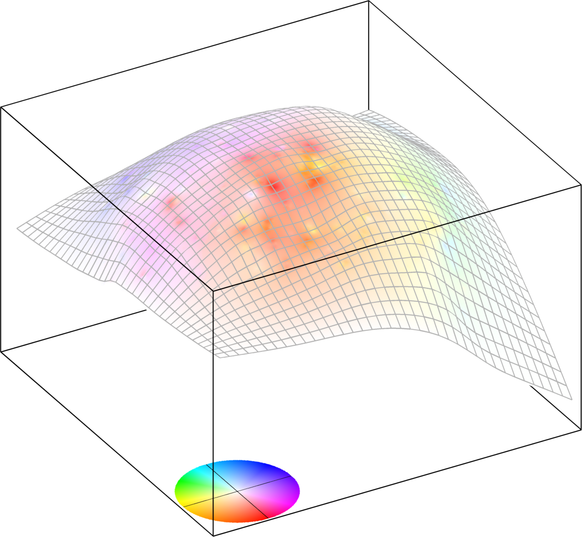} \\
 \includegraphics[width=.24 \textwidth]{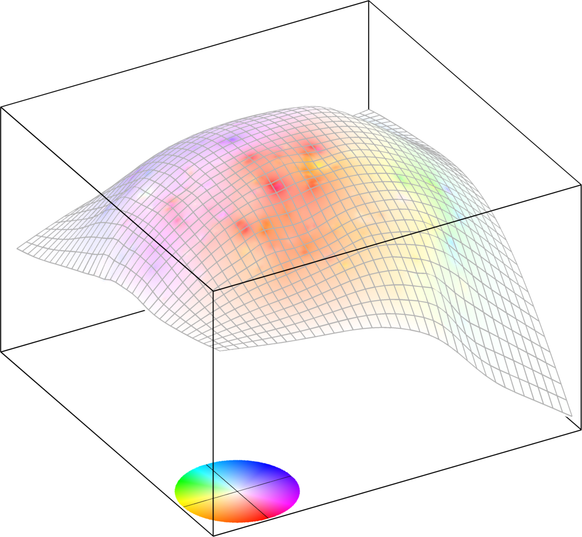}
\includegraphics[width=.24 \textwidth]{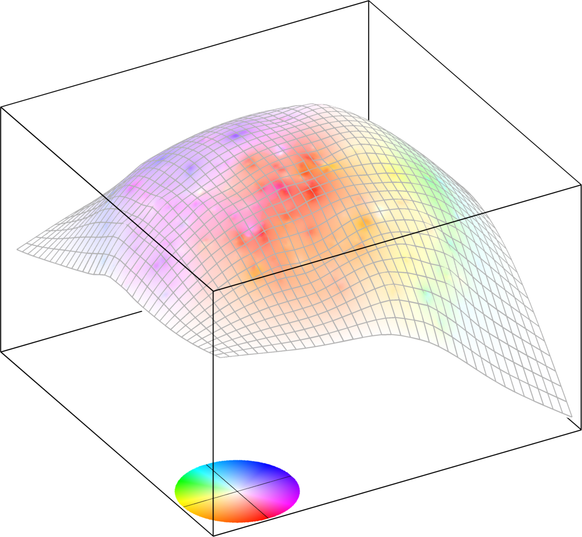}
\includegraphics[width=.24 \textwidth]{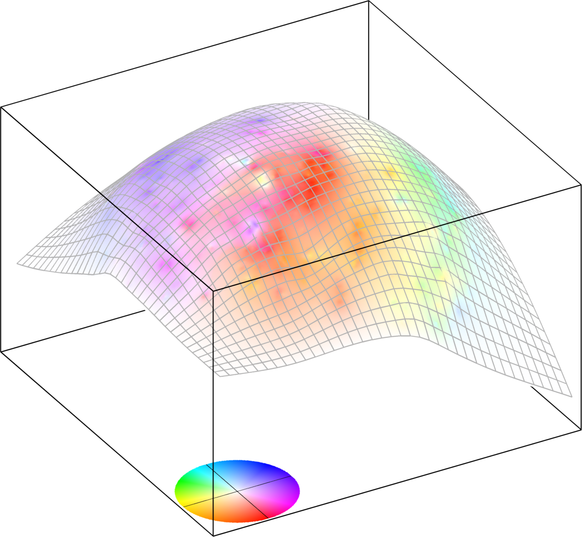}
\includegraphics[width=.24 \textwidth]{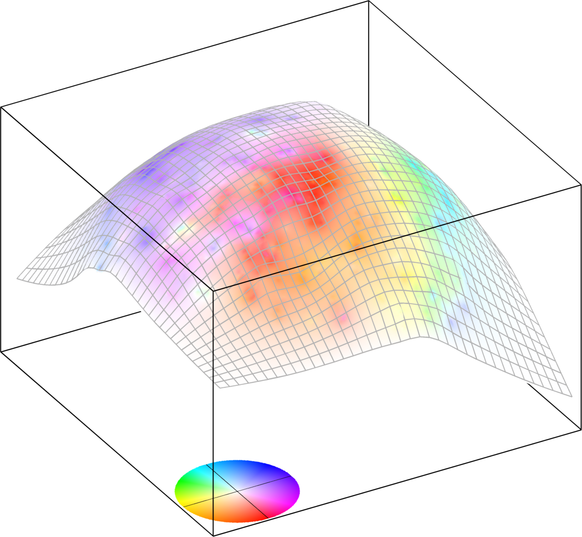}
 \caption{The color-coded optical flow vector field resulting from Experiment II at frames 1, 4, 7, 10, 11, 14, 17, 20.}
 \label{fig:e3colors}
\end{figure}

\begin{figure}
 \includegraphics[width=.24 \textwidth]{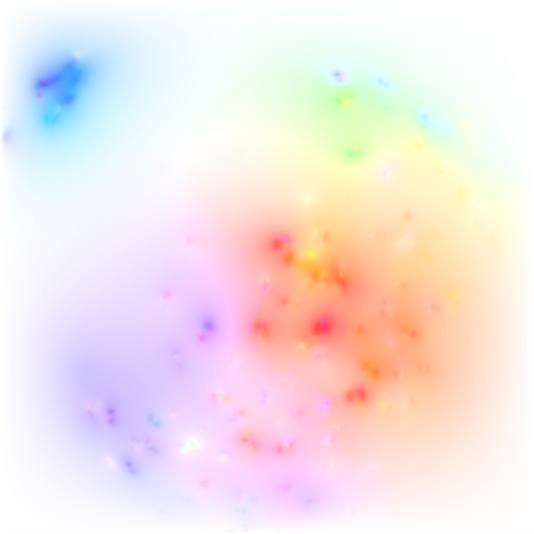}
 \includegraphics[width=.24 \textwidth]{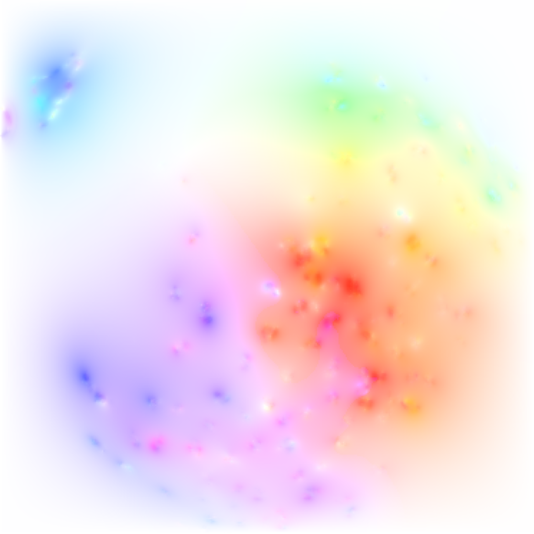}
 \includegraphics[width=.24 \textwidth]{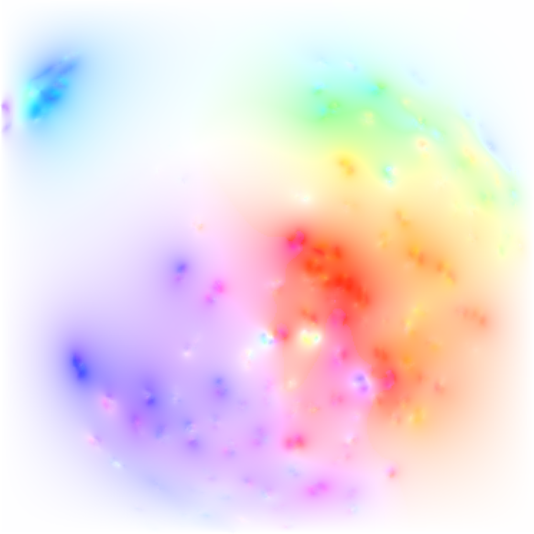}
 \includegraphics[width=.24 \textwidth]{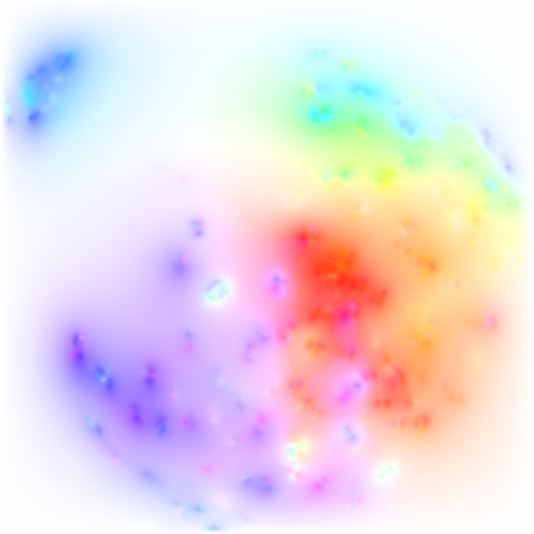}
\\
\includegraphics[width=.24 \textwidth]{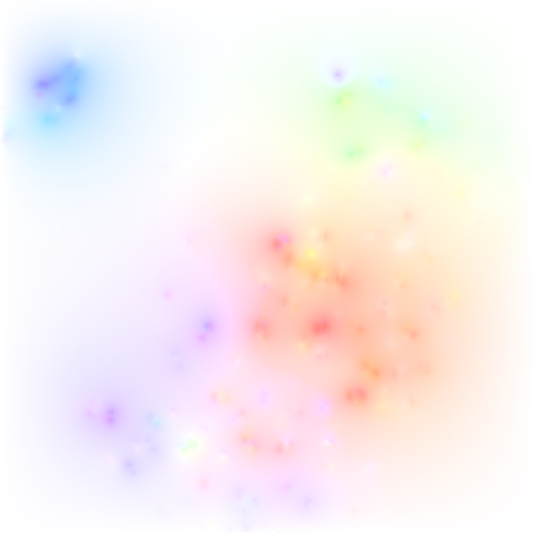}
\includegraphics[width=.24 \textwidth]{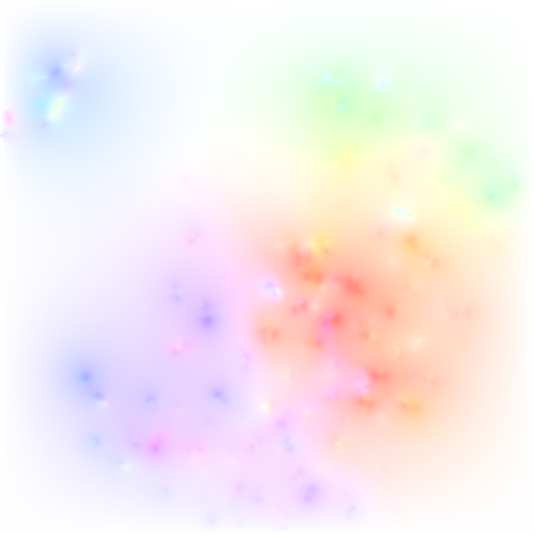}
\includegraphics[width=.24 \textwidth]{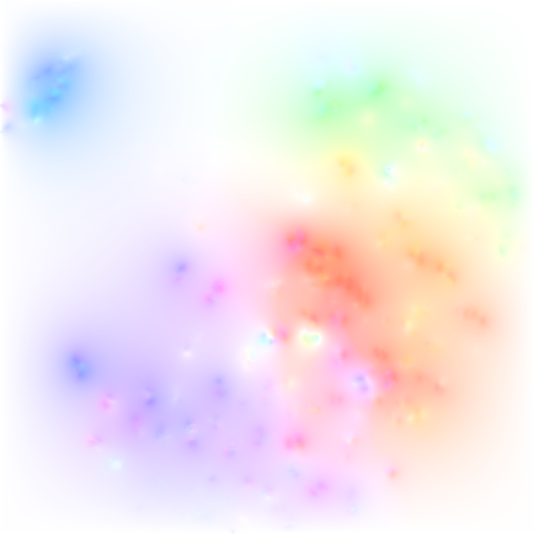}
\includegraphics[width=.24 \textwidth]{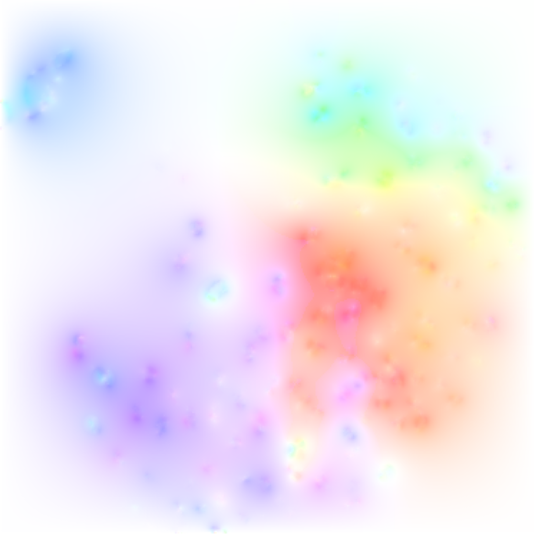}
\vspace{-0.2cm}
 \flushright{\includegraphics[width=.05 \textwidth]{figures/cw} }
 
 \caption{Comparison of the optical flow from Experiment II (frames 11 14, 17, 20) with the optical flow computed in the plane. First row: pull back optical flow from Fig.~\ref{fig:e3colors}, second row. Second row: optical flow computed in the plane. The color wheel is shown at the very bottom.}
\label{fig:e3comparison1}
\end{figure}

\paragraph{Experiment II, microscopy data.}
Finally, we test our model on real-world data. The image sequence under consideration in this section depicts a living zebrafish embryo during the gastrula period and has been recorded with a confocal laser-scanning microscope. The only visible feature in this dataset are the embryo's endodermal cells which, expressing a green fluorescent protein, proliferate on the surface of the embryo's yolk. Understanding and reconstructing cell motion during embryogenesis is a major topic in developmental biology and optical flow is one way to automatically extract this information \cite{AmaMyeKel13,AmaLemMosMcDWan14,KirLanSch14a,SchmShaScheWebThi13}. See \cite{KimBalKimUllSchi95} for a detailed account on the embryonic development of a zebrafish, and \cite{MegFra03} for more information on laser-scanning microscopy and fluorescent protein technology. 

The considered data do not depict the whole embryo but only a cuboid section of approximately $540\times 490\times 340 \,\mu m^3$. They have a spatial resolution of $512 \times 512\times 40$ voxels and the elapsed time between two consecutive frames is about four minutes. As in \cite{KirLanSch14a,SchmShaScheWebThi13} we avoid computational challenges by exploiting the fact that during gastrulation endodermal cells form a monolayer. This means they can be regarded as sitting on a two-dimensional surface. Therefore, by fitting a surface through the cells' positions, we can reduce the spatial dimension of the data by one. We refer to \cite{KirLanSch14a} on how this surface extraction was done.

In this particular experiment we apply our model to $21$ frames of the resulting 2D cell images with a resolution of $373 \times 373$ and again scaled to the unit interval. The extracted surface can be conveniently parametrized as the graph of a function $z(t,x_1,x_2)$ describing the height of the surface. That is, $f$ takes the form
\begin{equation*}
	f(t,x_1,x_2) = (x_1,x_2,z(t,x_1,x_2)).
\end{equation*}
In Fig.~\ref{fig:e3data} we show $I$, $\mathcal{M}_t$ and $\mathcal{I}$. The regularization parameters were set to $\alpha=10$, $\beta = 0$, $\gamma = 1$ and for the spatial boundaries we chose homogeneous Dirichlet boundary conditions. The GMRES solver converged faster this time and was terminated after the relative residual dropped below $10^{-3}$. Results are shown in Fig.~\ref{fig:e3colors}. In Fig.~\ref{fig:e3comparison1} we juxtapose the pulled back optical flow with the optical flow computed in the plane with the same parameters. Finally, we again compare the two vector fields by computing their angular error. This time we do so after push forward only, since for real-world data we are primarily interested in the vector field on the embedded surface.

\begin{figure}
	\begin{center}
	\includegraphics[width=.46 \textwidth]{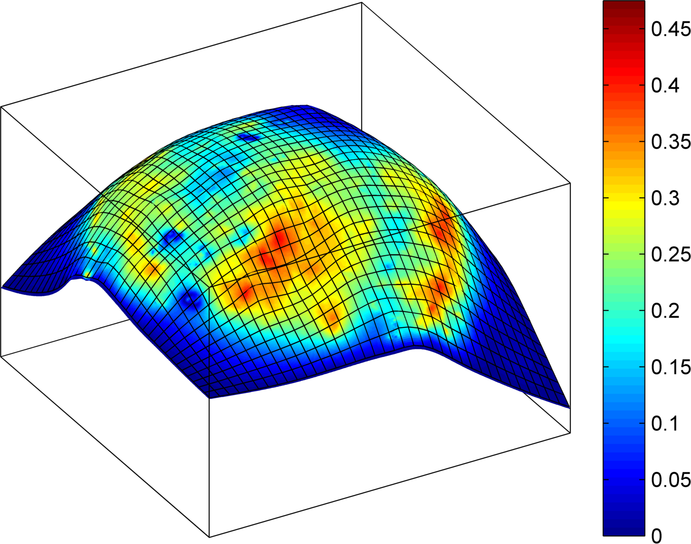}
	\caption{$\mathbb{R}^3$ angular errors at frame 20 between the optical flow computed on the zebrafish surface and the optical flow computed in the plane, but pushed forward to the same surface.}
	\label{fig:e3angerr}
	\end{center}
\end{figure}

\section{Conclusion}

Choosing a suitable regularization term is a major challenge in the
computation of the optical flow on a moving manifold when using
variational methods. The main question is how to incorporate the
structure of the manifold and of its movement into the regularization.
In this paper we have approached this problem from a purely differential
geometric point of view. We have constructed a Riemannian metric on the
time-dependent manifold in such a way that the paths of points on the
manifold are geodesics with respect to this metric. We have then used
a Horn--Schunck type quadratic regularization term with additional
time smoothing for the computation of the optical flow.
The experiments performed within this setting indicate the viability
of this approach and also show that using the manifold structure can
have a significant influence on the computed optical flow field.
Still, because of the usage of a quadratic regularization term
that is not adapted to the image structure, the resulting flow fields
tend to be oversmoothed.
The next step is therefore the extension to more complicated,
anisotropic regularization terms as discussed in~\cite{WeiSchn01a},
which may be more accurate for certain applications of optical flow.


\appendix

\section*{Appendix}
\section{Finding the parametrization}\label{ap:param}
So far we have assumed, that we are given the moving surface with a fixed parametrization. 
In applications this parametrization might be unknown, i.e., one might only observe the shape 
of the surface, but not its actual parametrization. Thus one will need to extract the parametrization from the observed data. In this part we will briefly sketch a possible approach to achieve this goal.
\begin{remark}
Note that the choice of parametrization will  have a tremendous influence on the resulting optical flow field. In particular one can  choose 
a parametrization, such that the optical flow field is almost zero. To achieve this one can take any fixed parametrization $f(t,\cdot)$ and solve the optical flow problem for this parametrization using small regularization parameters. 
Then one can use the resulting optical flow field $v$ to generate a path of diffeomorphisms $\varphi(t,\cdot)\in \operatorname{Diff}(M)$. Then the path $\tilde f(t,x)=f(t,\varphi(t,x))$ has the desired property.
\end{remark}
In the following we will assume that the evolution of the image has no influence on the evolution of the surface---the influence of the surface evolution on the image evolution is taken into account by the nature of the regularization term.
Furthermore we assume that we are given only the shape of the surface at each timepoint $t$, but not the actual parametrization, i.e., that we are given a path in the space of unparametrized, embedded surfaces; see \cite{Michor1980,Hamilton1982}
for a rigorous mathematical definiton of this infinite dimensional manifold.

At each time point $t$ we can now choose any parametrization of the surfaces yielding a path  of embeddings 
$$f: [0,T]\times M\mapsto \mathbb R^3.$$
Thus we have reduced the problem to finding the path of reparametrizations that best corresponds to the observed shape evolution. 

One way to tackle this problem is to define an
energy functional on the space of embeddings that incorporates the available information
on realistic shape evolutions.
In order to be independent of the initial parametrization of the path of surfaces we require 
that the energy functional is invariant under the action of the diffeomorphism group, i.e.,
$E(f(t,\varphi(x))=E(f(t,x))$, for all $\varphi \in \operatorname{Diff}(M)$. In this case, the energy functional on the space of parametrized surfaces induces 
an energy functional on the shape space of unparametrized surfaces.
Such a functional can be defined using a Riemannian metric, a Finsler type metric, or by some even more general Lagrangian, see 
e.g. \cite{Bauer2011b,Bauer2012d,Jermyn2011,Jermyn2012,Rumpf2011,Rumpf2012,Bauer2014}.

For the sake of simplicity, we will focus on the Riemannian case only, i.e.,
\begin{align*}
E(f)=\int_0^T G_f(f_t,f_t) dt, 
\end{align*}
where $G$ is some reparametrization invariant metric on the manifold of all embeddings. 

For historical reasons going back to Euler~\cite{Eul1765a}, these metrics are often represented via the corresponding inertia operator $L$:
$$G^L_f(f_t,f_t):=\int_M \langle f_t,L_f f_t\rangle \operatorname{vol}(g).$$
The simplest such metric is the reparametrization invariant $L^2$-metric -- or $H^0$--metric. This metric is induced by the operator $L=\operatorname{Id}$:
$$G_f^{0}(f_t,f_t)=\int_M \langle f_t,f_t\rangle \operatorname{vol}(g).$$
In order to guarantee that the bilinear form $G^L$ really induces a Riemannian metric, we require $L$ to be an elliptic pseudo-differential operator, that is symmetric and positive with respect to 
the $L^2$-metric. In addition we assume that $L$ is invariant under 
the action of the reparametrization group $\operatorname{Diff}(M)$. 
The invariance of $L$ implies that the induced metric $G^L$ is invariant 
under the action of $\operatorname{Diff}(M)$ as required. 
Using the operator $L$, one can include physical or biological model-parameters in the definition of the metric.

%

Now we want to find the optimal reparametrization of the initial path $f$ with respect to this energy functional. Therefore we have to solve the optimization problem:
$$\psi(t,x)=\underset{\varphi \in  C^{\infty}([0,T],\operatorname{Diff}(M))}{\operatorname{argmin}} E(f(t,\varphi(t,x)).
$$
Further expanding the energy functional using the invariance of the Riemannian metric yields:
\begin{align*}
E(f(t,\varphi(t,x))= \int_0^T G_f(f_t,f_t) &+ G_f(f_t,Tf(\varphi_t\circ\varphi^{-1})) \\&+ G_f(Tf(\varphi_t\circ\varphi^{-1}),Tf(\varphi_t\circ\varphi^{-1})) dt.
\end{align*}
\begin{remark}
 As an example we want to consider this functional for the $L^2$-metric.
Therefore we  decompose $f_t$ for each time point $t$ into a part that is normal to the surface $f$ and a part that is tangential:
$$f_t=Tf.f_t^\top+f_t^\bot.$$
Since these parts are orthogonal to each other -- w.r.t. the $L^2$--metric --
the energy functional reads as 
\begin{align*}
E(f(t,\varphi(t,x))= \int_M \langle f_t^{\bot},f_t^{\bot} \rangle \operatorname{vol}(g)+\int_M g(f_t^{\top}+\varphi_t\circ\varphi^{-1},f_t^{\top}+\varphi_t\circ\varphi^{-1}) \operatorname{vol}(g)
\end{align*}
This functional is minimal for
$$\varphi_t\circ\varphi^{-1}=-f_t^{\top}.$$
This however corresponds to a reparametrization $\varphi$ such that $\tilde f=f\circ\varphi$ consists only of a deformation in normal direction.
\end{remark}
\begin{remark}
 For a more general metric $G^L$ this will not hold anymore, since normal and tangential vector fields might not be orthogonal with respect to the $G^L$--metric.
Instead one can show that for the optimal path $\tilde f$ we will have that $L \tilde f_t$ is normal, cf.~\cite{Bauer2011b}.
\end{remark}

\section{Proof of Theorem~\ref{th:opt_coord}}

In the following we give a sketch of the derivation of the formulas
in Theorem~\ref{th:opt_coord}.

\begin{lemma}
The Christoffel symbols of the metric $\bar g$ have the form
given in ~\eqref{eq:Christoffel}.
\end{lemma}

\begin{proof}
This is a straight forward computation using the definition
of the Christoffel symbols as
\[
\bar \Gamma_{kl}^i=\frac12 \bar g^{im}\big(\bar g_{mk,l}+\bar g_{ml,k}-\bar g_{kl,m}\big)
\]
and the fact that the metric $\bar{g}$ and its inverse have the forms
\[
\bar{g} = 
\begin{pmatrix}
  \alpha^2 & 0 & 0 \\
  0 & g_{11} & g_{12} \\
  0 & g_{12} & g_{22}
\end{pmatrix}
\qquad\text{ and }\qquad
\bar g^{-1}=
\begin{pmatrix}
  \alpha ^{-2} &  0 &0\\
  0 & g^{11}&g^{12}\\
  0&g^{12}&g^{22}
\end{pmatrix},
\]
respectively, and
\[
g_{ij} = \langle \partial_i f, \partial_j f\rangle_{\R^3}.
\]
\end{proof}

\begin{lemma}
  The symbols $\bar{\omega}_{ik}^j$ have the form given in~\eqref{eq:connection}.
\end{lemma}

\begin{proof}
  The connection coefficients $\bar{\omega}_{ik}^j$ are defined as
  \[
  \bar \omega^j_{ik} = \left( \bar a^\ell_i \partial_\ell \bar a^m_k 
    + \bar a^\ell_i \bar a^n_k \bar\Gamma^m_{\ell n} \right) \bar a^h_j \bar g_{mh}\,.
  \]
  Moreover, the coordinates $\bar{a}_i^\ell$ have the form
  \[
  \bar{a}_i^\ell = 
  \begin{cases}
    \alpha^{-1} & \text{ if } i = \ell = 0,\\
    0 & \text{ if } i = 0 \text{ and } \ell \neq 0, \text{ or } i\neq 0 \text{ and } \ell = 0,\\
    a_i^\ell & \text{ if } i, \ell \neq 0.
  \end{cases}
  \]
  Using these facts and the form of the Christoffel symbols derived
  in~\eqref{eq:Christoffel}, the result follows from a straight forward calculation.
\end{proof}

\begin{lemma}\label{le:sim_grad}
  The $L^2$-gradient of the similarity term $\mathcal{S}$ in the energy 
  functional $\mathcal{E}$ can be written for $\bar{u} = (0,u^j X_j)$ as
  \[
  \operatorname{grad} \mathcal{S}(\bar{u})
  = 2\bigl(\partial_t I + \partial_\ell I a_m^\ell u^m\bigr)\partial_k I g^{ik} b_k^j X_j.
  \]
\end{lemma}

\begin{proof}
  As shown in Theorem~\ref{th:gradient}, the gradient of $\mathcal{S}$ has the form
  \[
  \operatorname{grad}\mathcal{S}(\bar{u}) 
  = 2\bigl(\partial_t I + g(\nabla^g I,u)\bigr)(0,\nabla^g I)\,.
  \]
  Denote now by $\tilde{u}^j$ the coordinates of $u$ with respect
  to $\partial_j$, that is, $u = \tilde{u}^j \partial_j$. Then
  \[
  g(\nabla^g I,\bar{u}) = (D_x I)\bar{u} = (\partial_\ell I)\tilde{u}^\ell.
  \]
  Moreover we have
  \[
  \tilde{u}^\ell = a_m^\ell u^m.
  \]
  Moreover, the coordinate expression of $\nabla^g I$ is $(\partial_k I)g^{ik}\partial_i$.
  Therefore we obtain
  \[
  \operatorname{grad}\mathcal{S}(\bar{u}) 
  = 2\bigl(\partial_t I + \partial_\ell I a_m^\ell u^m\bigr)\partial_k I g^{ik}\partial_i.
  \]
  Since $\partial_i = b_k^j X_j$, we obtained the claimed representation.
\end{proof}

\begin{lemma}\label{le:BL1}
  In the local coordinate frame $\bar X_0=\frac1{\alpha}\partial_t,\bar X_1,\bar X_2$  the Bochner Laplacian on the Riemannian manifold $(\bar M,\bar g)$ 
  of a vector field $\bar u$ satisfying  Neumann boundary conditions
  $$\bar \nabla_{\nu}\bar u\big|_{\partial \bar M}= \bar \nabla_{\partial_t}\bar u(\cdot,x)\big|_{0}^T=0$$
  is given by
  \begin{equation}\label{eq:BL}
    \Delta^B \bar u = \bar \nabla^* \bar \nabla \bar u
    =-\sum_{i=0}^2 \bar \nabla^2_{\bar X_i,\bar X_i} \bar u
    -\frac{\operatorname{Tr}(g^{-1}\partial_t g)}{2\alpha}\bar \nabla_{\bar X_0}\bar u\,.
  \end{equation}
\end{lemma}
\begin{proof}
  To calculate the expression of the Laplacian we have to compute the formula for the $L^2$-adjoint of the covariant derivative.  Taking two vector fields $\bar u$, $\bar v$ we have
  \begin{align*}
    \int_0^T\int_M \bar g(\Delta^B  u,\bar v)\vol(g)\,dt&= \int_0^T\int_M \bar g^1_1(\bar \nabla \bar u,\bar \nabla \bar v)\vol(g)\,dt
    \\&=\sum_{i=0}^2  \int_0^T\int_M \bar g(\bar \nabla_{\bar X_i} \bar u,\bar \nabla_{\bar X_i} \bar{v})\vol(g)\,dt.    
  \end{align*}
  Using $(\bar \nabla_{\bar X_i}\bar g)=0$ we obtain the following expression for the first summand ($i=0$): 
\begin{align*}
&\frac{1}{\alpha^2} \int_0^T\int_M \bar g(\bar \nabla_{\partial_t} \bar u,\bar \nabla_{\partial_t} \bar{v})\vol(g)\,dt    
\\&=
\frac{1}{\alpha^2} \int_0^T\int_M \partial_t \left(\bar g(\bar \nabla_{\partial_t} \bar u,\bar v)\right)\vol(g)\,dt-\int_0^T\int_M \frac{1}{\alpha^2} \bar g\left(\bar \nabla_{\partial_t}  \left(\bar \nabla_{\partial_t} \bar u\right),\bar v\right)\vol(g)\,dt
\\&=
\frac{1}{\alpha^2} \int_0^T\partial_t\left( \int_M  \bar g(\bar \nabla_{\partial_t} \bar u,\bar v)\vol(g)\right)\,dt-
\int_0^T\int_M \frac{1}{\alpha^2} \bar g(\bar \nabla_{\partial_t} \bar u,\bar v)\partial_t\vol(g)\,dt\\&\qquad-
\int_0^T\int_M \frac{1}{\alpha^2} \bar g(\bar \nabla^2_{\partial_t,\partial_t}   \bar u,\bar v)\vol(g)\,dt\,.
\end{align*}
 Using the variational formula \cite[Section 4.6]{Bauer2012a} $$\partial_t \operatorname{vol}(g)=\operatorname{Tr}(g^{-1}\partial_t g)\vol(g)$$ 
 yields
  \begin{align*}
    &\frac{1}{\alpha^2} \int_0^T\int_M \bar g(\bar \nabla_{\partial_t} \bar u,\bar \nabla_{\partial_t} \bar{v})\vol(g)\,dt    
    \\&\qquad=
    \frac{1}{\alpha^2} \left( \int_M  \bar g(\bar \nabla_{\partial_t} \bar u,\bar v)\vol(g)\right)\Big|_0^T
    \\&\qquad\qquad-
    \int_0^T\int_M \frac{1}{\alpha^2} \bar g(\bar \nabla_{\partial_t} \bar u,\bar v) \operatorname{Tr}(g^{-1}\partial_t g)\vol(g)  \,dt
    \\&\qquad\qquad-
    \int_0^T\int_M \frac{1}{\alpha^2} \bar g(\bar \nabla^2_{\partial_t,\partial_t}   \bar u,\bar v)\vol(g)\,dt\,.
  \end{align*}
Note that for Neumann boundary conditions  the first term in the above expression vanishes.
 
Since $M$ has no boundary, the other summands in the formula for $\Delta^B$ are similar but simpler: 
  \begin{align*}
    &\sum_{i=1}^2  \int_0^T\int_M \bar g(\bar \nabla_{\bar X_i} \bar u,\bar \nabla_{\bar X_i} \bar{v})\vol(g)\,dt\\
    &\qquad=
    \sum_{i=1}^2   \int_0^T\int_M  0-\bar g(\bar \nabla_{X_i}\left(\bar \nabla_{X_i} \bar u\right), \bar v)\vol(g)\,dt\\&=
    -\sum_{i=1}^2   \int_0^T\int_M \bar g(\bar \nabla^2_{X_i,X_i} \bar u, \bar v)\vol(g)\,dt\,.
  \end{align*}
Combining these equations we obtain the desired formula for $\Delta^B$.
\end{proof}

\begin{proof}[of Theorem~\ref{th:opt_coord}]
  We have already derived the representation of the Christoffel symbols
  and connection coefficients.
  
  Next, we will derive an explicit representation of the Bochner Laplacian
  in coordinates. To that end, we treat the two terms in~\eqref{eq:BL} separately. 
  For the second term we note that
  \[
  \bar{\nabla }_{\bar{X}_0}\bar{u} = (\bar{a}_0^m\partial_m u^j + u^m \bar{\omega}_{0m}^j) X_j\,,
  \]
  and thus we obtain
  \[
  -\frac{1}{2\alpha}\operatorname{Tr}(g^{-1}\partial_t g)\bar{\nabla }_{\bar{X}_0}\bar{u}
  = -\frac{1}{2\alpha} g^{ik} \partial_t g_{ik} u^m \bar{\omega}_{0m}^j\,.
  \]
  Moreover, we obtain from~\eqref{eq:Christoffel} that
  \[
  g^{ik}\partial_t g_{ik} = \bar{\Gamma}_{n0}^n.
  \]
  Hence, the first term becomes
  \begin{equation}\label{eq:t1}
    \frac{\operatorname{Tr}(g^{-1}\partial_t g)}{2\alpha}\bar \nabla_{\bar X_0}\bar u
    = -\frac{1}{2\alpha} \bar{\Gamma}_{n0}^n (\bar{a}_0^m\partial_m u^j + u^m\bar{\omega}_{0m}^j)X_j.
  \end{equation}

  For the second term, we compute
  \begin{equation}\label{eq:t2}
    \begin{aligned}
      \bar\nabla_{\bar{X}_i}\left(\bar\nabla_{\bar{X}_i} \bar{u}\right)
      &= \bar\nabla_{\bar{X}_i} \bigl(\bar{a}_i^\ell\partial_\ell u^j + u^k\bar{\omega}_{ik}^j\bigr)\bar{X}_j\\
      &= \Bigl(\bar{a}_i^m \partial_m\bigl(\bar{a}_i^\ell\partial_\ell u^j + u^k\bar{\omega}_{ik}^j\bigr)
      + \bigl(\bar{a}_i^\ell\partial_\ell u^m + u^k\bar{\omega}_{ik}^m\bigr)\bar{\omega}_{im}^j\Bigr)X_j\\
      &= \Bigl( (\bar{a}_i^m \partial_m \bar{\omega}_{ik}^j + \bar{\omega}_{ik}^m \bar{\omega}_{im}^j)u^k
      + \bar{a}_i^m \partial_m \bar{a}_i^\ell \partial_\ell u^j \\
      & \qquad\quad{} + 2\bar{a}_i^m \bar{\omega}_{ik}^j \partial_m u^k
      + \bar{a}_i^m \bar{a}_i^\ell \partial_{\ell m} u^j\Bigr)X_j.
    \end{aligned}
  \end{equation}

  Combining Lemma~\ref{le:sim_grad}, equations~\eqref{eq:t1} and~\eqref{eq:t2},
  and the fact that the gradient of $\norm{\bar{u}}_{0,\bar g}^2$ is simply
  $2u^j X_j$, we arrive, after dividing everything by two, at~\eqref{eq:main_eq}; equations~\eqref{eq:neuman_coord}
  are simply the Neumann boundary conditions in coordinate form.
\end{proof}

\vspace{1cm}

{\small
\begin{tabular}{rcl}
\textit{E-mail:} &\qquad &
\href{mailto.bauer.martin@univie.ac.at}{Bauer.Martin@univie.ac.at}\\
&&\href{mailto:markus.grasmair@math.ntnu.no}{Markus.Grasmair@math.ntnu.no}\\
&&\href{mailto:clemens.kirisits@univie.ac.at}{Clemens.Kirisits@univie.ac.at}\\
\end{tabular}
}

\end{document}